\documentclass[a4paper]{article}
\usepackage{latexsym}
\usepackage{amsmath}
\usepackage{amssymb}
\usepackage{enumerate}
\usepackage{amsthm} 
\usepackage{graphicx} 
\usepackage{amssymb,latexsym,times}
\usepackage{proof}
\usepackage{setspace}
\usepackage{array}

\newcolumntype{C}[1]{>{\centering\let\newline\\\arraybackslash\hspace{0pt}}m{#1}}
\newcommand{\len}{\textrm{len}}
\newcommand{\Dom}{\textrm{Dom}}
\newcommand{\Var}{\textrm{Var}}
\newcommand{\Fr}{\textrm{Fr}}

\newcommand{\N}{\mathbb{N}}

\newcommand{\Z}{\mathbb{Z}}
\newcommand{\Po}{\mathcal{P}}
\newcommand{\I}{\mathcal{I}}
\newcommand{\M}{\mathcal{M}}
\newcommand{\on}{\exists}
\newcommand{\ja}{\wedge}
\newcommand{\tai}{\vee}
\newcommand{\yli}{\overline}

\renewcommand{\a}{\alpha}
\renewcommand{\b}{\beta}

\newcommand{\ep}{\epsilon}
\newcommand{\de}{\delta}
\def\dep{=\!\!}
\newcommand {\indep}[3] {#2 ~\bot_{#1}~ #3}
\newcommand {\indepc}[2] {#1 ~\bot~ #2}

\newcommand{\sub}{\subseteq}\theoremstyle{plain}
\newtheorem{lau}[equation]{Theorem}
\newtheorem{lem}[equation]{Lemma}
\newtheorem{prop}[equation]{Proposition}
\newtheorem{kor}[equation]{Corollary}

\theoremstyle{definition}
\newtheorem{maa}[equation]{Definition}

\newtheorem{esim}[equation]{Example}

\title{Axiomatizing first-order consequences in \\independence logic}
\author{Miika Hannula}

\begin{document}
\maketitle

\begin{abstract}
Independence logic, introduced in \cite{erikjouko}, cannot be effectively axiomatized. However, first-order consequences of independence logic sentences can be axiomatized. In this article we give an explicit axiomatization and prove that it is complete in this sense. The proof is a generalization of the similar result for dependence logic introduced in \cite{juhajouko}.
\end{abstract}

\section{Introduction}

Independence logic \cite{erikjouko} is a recent variant of dependence logic that extends first-order logic by formulas
$$t_1 \bot_{t_3} t_2$$
where $t_i$ is a tuple of terms. The intuitive meaning of this formula is that the sets of values of $t_1$ and $t_2$ are independent of each other for a fixed value of $t_3$. Dependence logic \cite{kirja} adds to first-order logic formulas
$$\dep(t_1, \ldots ,t_n)$$
where $t_i$ is a term. Intuitively, this formula says that the values of $t_1, \ldots ,t_{n-1}$ determine the value of $t_n$. As the notions of dependence and independence are not interesting for single assignments, the semantics of these two logics are defined for sets of assignments, called teams. 

Historically these logics are preceded by partially ordered quantifiers (Henkin quantifiers) of Henkin \cite{henkin} and Independence-Friendly (IF) logic of Hintikka and Sandu \cite{jaakkogabriel}. Dependence logic is a variant of these two and equivalent in expressive power whereas independence logic is a bit more general formalism. Dependence logic sentences can be translated to existential second-order logic (ESO) sentences and vice versa. From the point of view of descriptive complexity theory, this means that dependence logic captures all the classes of \textit{models} in NP. Still, on the level of formulas, dependence logic is weaker in expressive power than ESO. Dependence logic formulas correspond to the ESO sentences that define a downwards closed class of teams \cite{juhajouko2}. 

This restriction does not apply to independence logic because it is not downwards closed. Galliani has showed that in expressive power independence logic is equivalent to ESO both on the level of formulas and sentences \cite{pietro}. It follows that all the NP classes of \textit{teams} are also definable in independence logic. 

In this article we consider only first-order consequences of independence logic. The reason for this restriction is that independence logic cannot be effectively axiomatized. In independence logic it is possible to describe infinity. Using this and going a little further, there is an independence logic formula $\Theta$ in the language of arithmetic saying that some elementary axioms of number theory fail or else some number has infinitely many predecessors. Now let $\phi$ be any first-order formula in the language of arithmetic. We show that the following claims are equivalent:
\begin{enumerate}
\item\label{true} $\phi$ is true in $(\N,+,\times,<)$.
\item\label{valid} $\Theta \tai \phi$ is valid (true in every model) in independence logic.
\end{enumerate}
Suppose first (\ref{true}) holds. Let $M$ be an arbitrary model of the language of arithmetic. If $M \not\models \Theta$, then we have that $M \cong (\N,+,\times ,<)$, and thus $M \models \phi$ when $M \models \Theta \tai \phi$. For the converse, suppose (\ref{valid}) holds. Since $(\N,+,\times,<) \not\models \Theta$, we have that $(\N, +,\times,<)\models \phi$.

The above shows that the truth in $(\N,+,\times,<)$ can be reduced to validity in independence logic. By Tarski's Undefinability of Truth, validity in independence logic is non-arithmetical. Therefore, independence logic cannot have any effective complete axiomatization.

The above result of non-axiomatizability holds also for dependence logic. However, this is not an end point of research here. There are at least two directions left. One is to modify the semantics in order to get a complete axiomatization. A good example of this is Henkin semantics for second-order logic, and for independence logic Galliani has taken this direction in \cite{pietro2}. Another is to only consider some fragment of a logic. In dependence logic this direction has been taken in \cite{juhajouko} where Kontinen and V\"a\"an\"anen present an explicit axiomatization for dependence logic and show that, although it cannot be fully complete, it is complete with respect to the first-order consequences of dependence logic sentences. Another interesting line is to consider atomic fragments. Although then in many cases we can directly apply  axiomatizations given in database theory. For instance, Armstrong's axioms for functional dependencies are also sound and complete for dependence atoms.\cite{arm} The atomic fragment of independence logic is however more complicated. Unlike with dependence atoms, the implication problem for independence atoms is undecidable, and therefore lacks finite axiomatization.\cite{herr1,herr2} Despite this, independence atoms have been axiomatized in \cite{hankon14} where completeness is obtained by using inclusion atoms and implicit existential quantification in the intermediate steps of derivations.

In this paper we will generalize the result  of \cite{juhajouko}   to independence logic. Although independence logic is strictly stronger than dependence logic, on the level of sentences these two logics coincide. Independence logic sentences can be translated to dependence logic sentences via ESO \cite{erikjouko}. So we already know that at least somehow this generalization can be done. 

Another background for this article is \cite{pietro} where Galliani studied variants of independence logic and different ways of defining semantics for these logics. One of these definitions will be both reasonable and useful for our purposes and will therefore be used in this paper. The semantics we will use is called LAX semantics in Galliani's work. Using it we can secure that only the variables occurring free in a formula will affect to the truth value of that formula. With LAX semantics we will be able to construct, for every independence formula, an equivalent formula in prenex normal form, and furthermore, an equivalent formula in a precise conjunctive normal form. This may be interesting in itself, although the constructions will be presented as parts of the completeness proof.

The structure of this paper is the following. In the next section we will go through some preliminaries that are necessary for this topic. In Section \ref{rulesit} the axioms and the rules of inference are introduced. In Section \ref{sound} we will show that our new deduction system is sound, and in Section \ref{complete} we will show that it is also complete in respect of first-order consequences of independence logic sentences. At the end of the paper some examples and further questions will be presented.

\section{Preliminaries}\label{pre}

In this section we introduce independence logic ($\I$) and go through some results that are needed in this paper. A few remarks on notations are needed. The most important one is that there will not be any notational distinction between tuples and singles. For example, $x$ can refer either to the single variable $x$ or to the tuple of variables $x=(x_1, \ldots , x_k)$. However, it is always mentioned in the text if we are considering tuples instead of singles at the time. Also if $x=(x_1, \ldots ,x_k)$ and $y=(y_1, \ldots ,y_l)$ are tuples of variables, then by $xy$ we denote the tuple $(x_1, \ldots ,x_k,y_1, \ldots ,y_l)$. If $A$ and $B$ are sets of tuples, then ${A}^{\frown} B$ denotes the set $\{ab \mid a \in A$ and $b \in B\}$.

\begin{maa}\label{def1}
Formulas of $\I$ are defined recursively as follows:
\begin{enumerate}
\item If $\phi$ is a first-order literal, then $\phi \in \I$.

\item If $t_1$, $t_2$ and $t_3$ are finite (or empty) tuples of terms, then $t_1 \bot_{t_3} t_2 \in \I$.

\item If $\phi, \psi \in \I$, then $\phi \tai \psi \in \I$ and $\phi \ja \psi \in \I$.

\item If $\phi \in \I$ and $x$ is a variable, then $\on x \phi \in \I$ and $\forall x \phi \in \I$.
\end{enumerate}
\end{maa}

Hence we allow negation only in front of first-order atoms. Also notice that in the independence atom, we allow any $t_i$ to be empty. In the case of $t_3 = \emptyset$, $t_1 \bot_{\emptyset} t_2$ is denoted by $t_1 \bot t_2$.

In order to define the semantics of $\I$, we first need to define the concept of a team. Let $M$ be a model. An assignment $s$ of $M$ is a finite mapping from a set of variables to the domain of $M$. (In this text $M$ can refer either to the model itself or its domain. It will be always clear from the context which one is under consideration.) Let $\{x_1, \ldots ,x_k\}$ be a set of variables. A team $X$ of $M$ with $\Dom (X)=\{x_1, \ldots ,x_k\}$ is a set of assignments $s$ of $M$ with $\Dom (s) = \{x_1, \ldots ,x_k\}$. The value of a term $t$ in an assignment $s$ is denoted by $t^M\langle s\rangle$. If $t=(t_1, \ldots ,t_l)$ where $t_i$ is a term, then by $t^M\langle s\rangle$ we denote $(t_1^M\langle s\rangle , \ldots ,t_l^M\langle s\rangle)$. By $s(a/x)$, for a variable $x$ and $a \in M$, we denote the assignment which (with domain $\Dom (s) \cup \{x\}$) agrees with $s$ everywhere except that it maps $x$ to $a$. Then by $X(M/x)$ we denote the duplicated team $\{s(a/x) \mid s \in X$ and $a\in M\}$. If $F:X \rightarrow \Po (M)$, then $X(F/x)$ denotes to the supplemented team $\{s(a/x) \mid s\in X$ and $a \in F(s)\}$. Note that it can be the case that $X(M/x)=X(F/x)$.

The set $\Fr (\phi)$ of free variables of a formula $\phi \in \I$ is defined as for first-order logic, except that we now have the new case
$$\Fr (t_1 \bot_{t_3} t_2) = \Var (t_1) \cup \Var (t_2) \cup \Var (t_3)$$
where $\Var (t_i)$ is the set of variables occurring in the term tuple $t_i$. If $\Fr (\phi) =\emptyset$, then we call $\phi$ a sentence.

Now we are ready to define the semantics of $\I$. In the definition, $M \models_s \phi$ refers to the Tarskian satisfaction relation of first-order logic.

\begin{maa}\label{def2}
Let $M$ be a model, $\phi \in \I$ and $X$ a team of $M$ such that $\Fr (\phi) \subseteq X$. The satisfaction relation $M \models_X \phi$  is defined as follows:
\begin{enumerate}
\item If $\phi$ is a first-order literal, then $M \models_X \phi$ iff $M \models_s \phi$ for all $s \in X$.
\item If $\phi = t_1 \bot_{t_3} t_2$, then $M \models_X \phi$ iff for all $s,s' \in X$ with $t_3^M \langle s\rangle = t_3^M \langle s'\rangle$, there is some $s'' \in X$ such that $t_1^M\langle s'' \rangle t_3^M \langle s''\rangle = t_1^M\langle s \rangle t_3^M \langle s\rangle$ and $t_2^M \langle s'' \rangle = t_2^M \langle s' \rangle$.

\item If $\phi = \psi \tai \theta$, then $M \models_X \phi$ iff $M \models_Y \psi$ and $M \models_Z \theta$ for some $Y,Z \subseteq X$, $Y\cup Z=X$. 

\item If $\phi = \psi \ja \theta$, then $M \models_X \phi$ iff $M \models_X \psi$ and $M \models_X \theta$.

\item If $\phi = \exists x \psi$, then $M \models_X \phi$ iff $M \models_{X(F/x)} \psi$ for some $F: X \rightarrow \Po (M)\setminus\{\emptyset\}$.

\item If $\phi = \forall x \psi$, then $M \models_X \phi$ iff $M \models_{X(M/x)} \psi$.
\end{enumerate}
\end{maa}

In the case of $t = \emptyset$ occurring in an independence atom, we let $t^M\langle s\rangle = t^M \langle s'\rangle$ for every $s,s' \in X$. Therefore,

$$M \models_X \emptyset \bot_{t_3} t_2 \textrm{ and } M \models_X t_1 \bot_{t_3} \emptyset \textrm{ for all } X.$$

If we are verifying a formula of the form $\forall x_1 \ldots \forall x_k \phi$, then instead of writing $X(M/x_1)\ldots(M/x_k)$, we will often use the abbreviation $X(M^k/x_1\ldots x_k)$. Also for verifying a formula of the form $\on x_1 \ldots \on x_k \phi$, we have to find witnessing functions $F_1: X \rightarrow \Po(M)\setminus\{\emptyset\},\ldots ,F_k:X(F_1/x_1)\ldots (F_{k-1}/x_{k-1}) \rightarrow \Po(M)\setminus\{\emptyset\}$ such that
$$M \models_{X(F_1/x_1)\ldots (F/x_k)} \phi.$$ 
Clearly in this case it is equivalent to find a single function $F: X \rightarrow \Po(M^k)\setminus\{\emptyset\}$ such that 
$$M \models_{X(F/x_1\ldots x_k)} \phi.$$

An immediate consequence of Definition \ref{def2} is that first-order formulas are flat in the following sense (the proof  is a straightforward structural induction).

\begin{prop}\label{flat}
Let $M$ be a model, $\phi$ a first-order formula and $X$ a team of $M$ such that $\Fr (\phi) \subseteq X$. Then the following are equivalent:
\begin{itemize}
\item $M \models_X \phi$,
\item $M \models_{\{s\}} \phi$ for all $s \in X$,
\item $M \models_s \phi$ for all $s \in X$.
\end{itemize}
\end{prop}

In addition to independence atoms, there are also many other type of atomic formulas that are relevant in team semantics setting. Dependence atom was already introduced but also inclusion and exclusion atoms will be useful for our purposes. The syntax of these atoms is the following:

\begin{itemize}
\item dependence: $\dep (t_1, \ldots ,t_n)$, where $t_1, \ldots ,t_n$ is a term.
\item inclusion: $t_1 \subseteq t_2$, where $t_1$ and $t_2$ are tuples of terms of the same length.
\item exclusion: $t_1 \mid t_2$,  where $t_1$ and $t_2$ are tuples of terms of the same length.
\end{itemize}

The semantics of these atoms is defined as:

\begin{itemize}
\item dependence: $M \models_X \dep(t_1, \ldots ,t_n)$ iff for all $s,s'\in X$ with $t_1^M \langle s\rangle =t_1^M \langle s'\rangle, \ldots ,t_{n-1}^M \langle s\rangle = t_{n-1}^M \langle s'\rangle$, it holds that $t_n^M\langle s\rangle =t_n^M \langle s'\rangle$.

\item inclusion:  $M \models_X t_1 \subseteq t_2$ iff for every $s\in X$, there is $s'\in X$ such that $t_1^M\langle s\rangle =t_2^M\langle s'\rangle$.

\item exclusion:  $M \models_X t_1 \mid t_2$ iff for every $s,s'\in X$ , $t_1^M\langle s\rangle \neq t_2^M\langle s'\rangle$.
\end{itemize}

If we replace independence atom with one these atoms in Definition \ref{def1}, then the resulting logic is called dependence logic, inclusion logic or exclusion logic.

Consider first dependence logic. Dependence atom $\dep (t_1 , \ldots ,t_n)$ express functional dependence between $t_n$ and the tuple $t_1 \ldots t_{n-1}$, and it can be expressed in independence logic as
$$t_n \bot_{t_1\ldots t_{n-1}} t_n.$$
For the other direction, there is no translation of independence atom in dependence logic. Dependence logic is downwards closed (meaning that $M \models_Y \phi$ whenever $M \models_X \phi$ and $Y \subseteq X$) whereas independence logic is not. Consider for example independence atom $x \bot y$. This atom is true for the team
\begin{equation}\label{X}\begin{array}{c|c|c}
   &x  &y\\
   \hline
s_0&0  &0\\
\hline
s_1&0  &1\\
\hline
s_2&1  &0\\
\hline
s_3&1  &1
\end{array}
\end{equation}
but not for the team
\begin{equation}\label{X'}
\begin{array}{c|c|c}
   &x   &y\\
   \hline
s_0&0  &0\\
\hline
s_1&0  &1\\
\hline
s_2&1  &0
\end{array}
\end{equation}
Thus it cannot be expressed in dependence logic, and hence independence logic is expressively strictly stronger on the level of formulas. On the level of sentences though, these logics coincide \cite{erikjouko}.

One should also mention that the semantics of all the other dependence logic formulas are not normally defined entirely the same way as we did in Definition \ref{def2} for independence logic. There is usually one exception concerning existential formulas. For $\on x \phi$, it is usually required that each value of the function $F: X \rightarrow \Po (M)\setminus\{\emptyset\}$ is singleton. Still, dependence logic is downwards closed with both semantics, and thus with Axiom of Choice, these two definitions coincide.

A direct consequence of the example above is that we cannot adopt the rule $\forall x \phi \vdash \phi$ into our inference system because it is not sound for independence logic. If $M$ is a model with domain $\{0,1\}$, and $X$ is the team (\ref{X'}), then $X(M/x)$ is the team (\ref{X}), and thus $M \models_X \forall x x \bot y$ and $M \not\models_X x\bot y$.

Consider then inclusion and exclusion atoms. Galliani has showed that inclusion/exclusion logic (first-order logic added with inclusion and exclusion atoms) is translatable to independence logic and vice versa \cite{pietro}. There the following independence logic translation of inclusion atom was presented.

\begin{prop}[\cite{pietro}]\label{pietron}
Let $t_1$ and $t_2$ be tuples of terms of the same
length. Then the inclusion atom $t_1 \subseteq t_2$ is equivalent to the independence formula
$$\forall v_1\forall v_2\forall z((\neg z = t_1\ja \neg z = t_2)\tai(\neg v_1 = v_2\ja \neg z = t_2)\tai((v_1 = v_2\tai z = t_2)\ja z \bot v_1v_2))$$
where $v_1$ and $v_2$ are variables and $z$ is a variable tuple of the same length than $t_i$, and none of the variables in $v_1v_2z$ occur in $t_1t_2$.
\end{prop}

We will use dependence and inclusion atoms in our deduction system, and there every such an occurrence should be understood as an independence logic translation of the form introduced here.

Before going to the proof, one important result need yet to be introduced.

\begin{maa}
Let $T$ be a set  of formulas of independence logic with only finitely many free variables. The formula $\phi$ is a logical consequence of $T$,
$$T \models \phi,$$
if for all models $M$ and teams $X$, with $\Fr (\phi) \cup \bigcup_{\psi \in T} \Fr (\psi) \subseteq \Dom (X)$, and \\$M \models_X T$, we have $M \models_X \phi$. The formulas $\phi$ and $\psi$ are logically equivalent,
$$\phi \equiv \psi,$$
if $\phi \models \psi$ and $\psi \models \phi$.
\end{maa}

Let $X$ be a team with domain $\{x_1, \ldots ,x_k\}$ and $V \subseteq \{x_1, \ldots ,x_k\}$. Then by $X \upharpoonright V$ we denote the team $\{s \upharpoonright V \mid s \in X\}$. If $u$ is a tuple of variables such that $\Var (u)=V$, then by $X \upharpoonright u$ we denote the team $X \upharpoonright V$. The following result is important \cite{pietro}.

\begin{prop}[Locality]\label{loc}
Suppose $V \supseteq \Fr (\phi)$. Then $M \models_X \phi$ iff $M \models_{X\upharpoonright V} \phi$.
\end{prop}

For a logic in team semantics setting, this is not an obvious fact. IF logic lacks this property, and the same holds for independence logic if the semantics of $\on x \phi$ is defined in the standard dependence logic way (requiring that the witnessing $F$ maps the assignments of $X$ to singletons of $\Po(M)$).

\section{A system of natural deduction}\label{rulesit}

In this section we introduce inference rules that allow us to derive all the first-order consequences of sentences of independence logic. Many of the rules below are just the same than the dependence logic rules introduced in \cite{juhajouko}. Still some major differences occur in this system partly due the semantic differences between independence and dependence atomic formulas and partly due the fact that independence logic is not downwards closed.

The rules we are about to adopt are listed below in Figure \ref{ded}, Figure \ref{dod} and Definition \ref{rules}. Figure \ref{ded} presents the usual inference rules of first-order logic with some restrictions, and Figure \ref{dod} presentes rules for inclusion atoms which are here thought of as independence logic translations of the form given in Proposition \ref{pietron}. If $A$ is a formula, $t=(t_1, \ldots ,t_n)$ is a tuple of terms and $x = (x_1, \ldots x_n)$ is  a tuple of variables, then $A(t/x)$ denotes the formula $A$ where all the free occurrences of $x_i$ are replaced by $t_i$. When using this notation we presume that no variable in $t_i$ becomes bound in the substitution.

\begin{figure}
\center\begin{tabular}{|c|c|c|}\hline
&& \\
{\small
Operation} & {\small Introduction} & {\small Elimination
} \\
&& \\
 \hline
&&\\
Conjunction
&
$
\infer[{\mbox{\tiny$\wedge$ I}}]{A \wedge B}{A & B}$
&
$\infer[{\mbox{\tiny$\wedge$ E}}]A{A\wedge B}
\qquad
\infer[{\mbox{\tiny$\wedge$ E}}]B{A\wedge B}
$\\&&\\
\hline
&&\\
Dis\-junc\-tion
&
$
\infer[{\mbox{\tiny$\vee$ I}}]{A \vee B}{A}
\qquad
\infer[{\mbox{\tiny$\vee$ I}}]{A \vee B}{B}$
&
$\infer[{\mbox{\tiny$\vee$ E}}]{C}{
       A\vee B
       &
                                   \infer*{C}{
               [A]
       }
       &
      \infer*{C}{
               [B]
       }
       }
$

\\&&\\
&
&Condition 1.\\
&&\\
\hline

&&\\
Ne\-ga\-tion 
 
&
$
\infer[{\mbox{\tiny$\neg$ I}}]{\neg A}{
      \infer*{B\wedge\neg B}{
               [A]
       }
       }
$
&

$
\infer[{\mbox{\tiny$\neg$ E}}]{A}{\neg\neg A}
$
\\&&\\

&Condition 2.
&Condition 3. 
\\
&&\\
\hline

&&\\
Uni\-ver\-sal quantifier
&
$
\infer[{\mbox{\tiny$\forall$ I}}]{\forall x_iA}{
      A
       }

$
&
$
\infer[{\mbox{\tiny$\forall$ E}}]{A(t/x_i)}{\forall x_i A}
$
\\&&\\
&Condition 4.
&Condition 3.
\\
&&\\
\hline
&&\\
Existential quantifier
&
$
\infer[{\mbox{\tiny$\exists$ I}}]{\exists x_iA}{
      A(t/x_i)
       }

$
&
$
\infer[{\mbox{\tiny$\exists$ E}}]{B}{
       \exists x_iA
       &
                                   \infer*{B}{
               [A]
       }
       }$
\\&&\\
&
&Condition 5.\\
&&\\
\hline
\multicolumn{3}{|l|}{  }\\
\multicolumn{3}{|l|}{\ Condition 1. $C$ and any non-discharged formula used in a derivation of $C$ must be
}\\
\multicolumn{3}{|l|}{\ first-order.}\\

\multicolumn{3}{|l|}{\ Condition 2. $A,B$, and any non-discharged formula used in the derivation of $B\ja \neg B$
}\\
\multicolumn{3}{|l|}{\ must be first-order.}\\
\multicolumn{3}{|l|}{\ Condition 3. $A$ is first-order.
}\\
\multicolumn{3}{|l|}{\ Condition 4. The variable \(x_i\) cannot appear free  in any non-discharged assumption
}\\
\multicolumn{3}{|l|}{\ used in the derivation of \(A\).}\\
 \multicolumn{3}{|l|}{\ Condition 5. The variable
 \(x_i\) cannot appear free  in \(B\) and in any non-discharged  
}\\
\multicolumn{3}{|l|}{\ assumption used in the derivation of  \(B\),
except in \(A\).}\\
\multicolumn{3}{|l|}{  }\\
\hline
\end{tabular}
\caption{The first set of rules.\label{ded}}
\end{figure}

\begin{figure}
\center
\begin{tabular}{|C{3.7cm}|C{3.7cm}|C{3.7cm}|}\hline

&&\\
&&\\ 
 $\infer{t \sub t}{}$&\infer{t_0 \sub t_2}{t_0 \sub t_1 \quad t_1 \sub t_2}& \infer{t_{i_1}\ldots t_{i_l} \sub t'_{i_1}\ldots t'_{i_l}}{t_1 \ldots t_n \sub t'_1 \ldots t'_n}\\
&&\\ 

&&Condition 5. \\
&&\\
Reflexivity& Transitivity &Projection and permutation\\\hline
\multicolumn{3}{|l|}{  }\\
\multicolumn{3}{|l|}{\ Note: $t,t_1,t'_1 \ldots ,t_n,t'_n$ are tuples of terms.
}\\
\multicolumn{3}{|l|}{\ Condition 5. The indices $i_1, \ldots ,i_l$ are from $\{1, \ldots ,n\}$.
}\\\hline

\end{tabular}
\caption{The second set of rules.\label{dod}}
\end{figure}

\begin{maa}\label{rules}
\begin{enumerate}
\item\label{rule1} Disjunction substitution: 
\[
\infer[]{A\vee C}{
        A\vee B
       &
                                   \infer*{C}{
               [B]
       }
       }\]
where the prerequisite for applying this rule is that any non-discharged assumption used in the derivation of $C$ must be a first-order formula.

\item\label{rule2}  Commutation and associativity of disjunction: 
\[\infer{A\vee B}{B\vee A}\hspace{19mm}\infer{A\vee(B\vee C)}{(A\vee B)\vee C}\]

\item\label{rule3} Extending scope: 
\[\infer{ \forall x ((A \ja x \bot y) \vee B)}{\forall x A \vee B}   \]
 where $y$ is a tuple listing the variables in $\Fr (A \tai B)-\{x\}$ and the prerequisite for applying this rule is that $x$ does not appear free in $B$.

\item\label{rule4} Extending scope:
\[\infer{ \on x (A \vee B)}{\on x A \vee B}   \]
 where the  prerequisite for applying this rule is that $x$ does not appear free in $B$.

\item\label{rule5} Universal substitution: \[
\infer[]{\forall y B}{\forall x A        
       &
                                   \infer*{B}{A(y/x)              
       }
       }\]
where the prerequisite for applying this rule is that $y$ does not appear free in $\forall x A$ and in any non-discharged assumption used in the derivation of $B$, except in $A(y/x)$.
 
\item\label{rule6} Independence distribution: Let
\begin{equation}
A = \on x_0 (\bigwedge_{1 \leq i \leq m} u_i \bot_{w_i} v_i \ja C)
\end{equation}
and
\begin{equation}
B = \on x_1 (\bigwedge_{m+1 \leq i \leq m + n} u_i \bot_{w_i} v_i \ja D)
\end{equation}
be formulas where $x_0$ is a tuple of variables that do not appear in $B$; $x_1$ is a tuple of variables that do not appear in $A$; $u_i$, $v_i$ and $w_i$ are tuples of bound variables; $C$ and $D$ are first-order formulas. 

Let
\begin{align*}
E=&\forall \a \forall \b \on x_0\on x_1 \on z_0\on z_1\on r [\bigwedge_{1 \leq i \leq m+n}  u_i \bot_{w_i r} v_i \ja \bigwedge_{i = 0,1}  \dep(z_i) \ja \\
&(\neg  z_0 = z_1 \tai \a = \b) \ja ((C  \ja r = z_0) \tai (D \ja r = z_1))]
\end{align*}
where $\a$, $\b$, $z_0$, $z_1$ and $r$ are variables that do not appear in formula $A \tai B$.
Then
\[\infer{E}
{A \tai B}\]
Note that the logical form of this rule is 
\[\infer{\begin{array}{l}
\forall \a \forall \b \on x_0\on x_1 \on z_0\on z_1\on r [\bigwedge_{1 \leq i \leq m+n}  u_i \bot_{w_i r} v_i \ja \bigwedge_{i = 0,1}  \dep(z_i) \ja \\
(\neg  z_0 = z_1 \tai \a = \b) \ja ((C  \ja r = z_0) \tai (D \ja r = z_1))]
\end{array}
}
{\on x_0 (\bigwedge_{1 \leq i \leq m} u_i \bot_{w_i} v_i \ja C) \tai \on x_1 (\bigwedge_{m+1 \leq i \leq m + n} u_i \bot_{w_i} v_i \ja D)}\]

\item\label{rule7} Independence introduction:
\[\infer{\forall y \on x (A \ja x \bot_z y)}{\on x \forall y A}\]
where $z$ is a tuple listing the variables in $\Fr (A)-\{x,y\}$.

\item\label{rule8} Inclusion compression:
\[\infer{ A(y/x)}{y \sub x \quad A}   \]
 where $x$ and $y$ are tuples of distinct variables, $x$ lists $\Fr(A)$, and the prerequisite for applying this rule is that $A$ is first-order.

\item\label{rule9} Independence elimination: \[
\infer[]{ \on x_2( x_2 \sub x \ja (A \tai u_0v_1w_0 = u_2v_2w_2)) }{\indep{w}{u}{v}  \quad x_0 \sub x \quad x_1 \sub x }\]
where  
$A=  \left\{\begin{array}{l l}
    \perp & \quad \textrm{if $w$ is empty,}\\
    \neg w_0 = w_1& \quad \textrm{otherwise,}\\
  \end{array}\right.$\\

and $u$, $v$ and $w$ are tuples of variables from $x$, and $z_i = z(x_i/x)$, for $z \in \{u,v,w\}$ and $i \in \{0,1,2\}$.

\item\label{rule10} Tuple introduction: \[
\infer[]{ \begin{array}{l}\forall x \on y \big ( B\ja \\
\forall u \on v \on x' \on y' (A(uv/xy) \ja x'y'=xy \ja x'y' \sub uv)\big ) \end{array}}{\forall x \on y (A \ja B)} \]
where $xyuvx'y'$ is a tuple of distinct variables such that $x,x',u$ are of the same length, $y,y',v$ are of the same lenght, and $xy$ lists $\Fr(A)$.

\item\label{rule11} Identity axiom:
If $x$ is a variable, then $x=x$ is an axiom.

\item\label{rule12} Identity rule:
If $x$ and $y$ are variables, then we let
\[\infer{y=x}{x=y}\]

\item\label{rule13} Identity rule:
If $t$ is a term and $x$ and $y$ are variables, then we let
\[\infer{t(x/y)=t}{x=y}\]

\item\label{rule14} Identity rule: 
If $A$ is a formula and $x$ and $y$ are variables, then we let
\[\infer{A(x/y)}{A \ja x=y}\]
\end{enumerate}
\end{maa} 

Disjunction elimination rule is not sound for independence logic, so we introduce rules 1-4 for disjunction. 
Also similar rules for conjunction are easily derivable in this system with an exception that we can derive the correspondent of rule 3 without this new independence atom $x \bot y$ occurring in the derived formula. As mentioned before, universal elimination rule does not hold for independence logic, so we introduce rule 5 here which is also derivable in first-order logic. Rules 3, 4, 6 and 7 preserve logical equivalence. Also note that rule \ref{rule9} is analogous to the chase rule of independence and inclusion atoms  in \cite{hankon14}.

\section{The Soundness Theorem}\label{sound}

In this section we will show that the previous system of natural deduction is sound. First we prove that rules 3, 4, 6 and 7 (plus the conjunctive versions of rules 3 and 4 which are denoted by 3' and 4') preserve logical equivalence.

\begin{lem}[Rules 3, 4, 7 and the conjunctive versions 3' and 4']\label{saannot} The following equivalences hold for formulas of independence logic:
\begin{enumerate}
\item[(3)] $\forall x ((\varphi \ja y \bot x) \tai \psi) \equiv \forall x \varphi \tai \psi$ if $x$ does not occur free in $\psi$ and $y$ is a tuple listing all the variables in $\Fr (\varphi \tai \psi) - \{x\}$.
\item[(3')] $\forall x (\varphi \ja \psi) \equiv \forall x \varphi \ja \psi$ if $x$ does not occur free in $\psi$.
\item[(4)] $\on x(\varphi \tai \psi) \equiv \on x \varphi \tai \psi$ if $x$ does not occur free in $\psi$.
\item[(4')] $\on x(\varphi \ja \psi) \equiv \on x \varphi \ja \psi$ if $x$ does not occur free in $\psi$.
\item[(7)] $\forall x \on y ( \phi \ja x \bot_{z} y) \equiv \on y \forall x \phi$ if $z$ is a tuple listing all the variables in $\Fr (\phi) - \{x,y\}$.
\end{enumerate}
\end{lem}
\begin{proof}
\begin{enumerate}
\item[(3)] By locality, it is enough to prove the equivalence for models $M$ and teams $X$ such that $\Dom (X) = \Fr (\varphi \tai \psi) - \{x\}$. So assume that $M \models_X \forall x ((\varphi \ja y \bot x) \tai \psi)$. Then we can find $Y,Z \subseteq X(M/x)$, $Y \cup Z = X(M/x)$, such that $M \models_Y \varphi \ja y \bot x$ and $M \models_Z \psi$. There are two options:\\
$(i)$ For all $s \in X$, there is some $a \in M$ such that $s(a/x) \in Z$. Then by locality, $M \models_X \psi$ and therefore $M \models_X \forall x \varphi \tai \psi$.\\
$(ii)$ For some $s \in X$, $s(a/x) \in Y$ for all $a \in M$. Then because $M \models_Y y \bot x$, we conclude that $Y(M/x) = Y$, and hence $M \models_Y \forall x \varphi$. If $Y'= Y \upharpoonright \Dom (X)$ and $Z' = Z \upharpoonright \Dom (X)$, then by locality, $M \models_{Y'} \forall x \varphi$ and $M \models_{Z'} \psi$. Now $X = Y' \cup Z'$, so we conclude that $M \models_X \forall x \varphi \tai \psi$.

For the converse, assume that $M \models_X \forall x \varphi \tai \psi$. Let $Y,Z \subseteq X$, $Y \cup Z = X$, be such that $M \models_{Y(M/x)} \varphi$ and $M \models_Z \psi$. Clearly $M \models_{Y(M/x)} \varphi \ja y \bot x$, and by locality, $M \models_{Z(M/x)} \psi$. So $M \models_{X(M/x)} (\varphi \ja y \bot x) \tai \psi$ and hence $M \models_X \forall x ((\varphi \ja y \bot x) \tai \psi)$.
\item[(3')] Follows from locality of the semantics.
\item[(4)] If $M \models_X \on x (\varphi \tai \psi)$, and $F : X \rightarrow \mathcal{P}(M)\setminus\{\emptyset\}$ is such that $M \models_{X(F/x)} \varphi \tai \psi$, then we can find $Y,Z \subseteq X(F/x)$, $Y \cup Z = X(F/x)$, so that $M \models_Y \varphi$ and $M \models_Z \psi$. Define 
$$Y' = \{s \in X \mid s(a/x) \in Y\textrm{ for some }a \in F(s)\}$$
 and 
$$Z' = \{s \in X \mid s(a/x) \in Z\textrm{ for some }a \in F(s)\}.$$ Then $M \models_{Z'} \psi$, and if $F' : Y' \rightarrow \mathcal{P}(M)\setminus\{\emptyset\}$ is the function $s \mapsto \{a \in F(s) \mid s(a/x) \in Y\}$, then $M \models_{Y'(F'/x)} \varphi$ and thus $M \models_{Y'} \on x \varphi$. Hence $M \models_X \on x \varphi \tai \psi$.

If $M \models_X \on x \varphi \tai \psi$, then for some $Y,Z \subseteq X$, $Y \cup Z = X$, $M \models_Y \on x \varphi$ and $M \models_Z \psi$. If $F : Y \rightarrow \mathcal{P}(M)\setminus\{\emptyset\}$ is such that $M \models_{Y(F/x)} \varphi$, choose $F': X \rightarrow \mathcal{P}(M)\setminus\{\emptyset\}$ so that $F' \upharpoonright Y = F$ and $F' \upharpoonright (X - Y)$ is some constant function. Then $Y(F'/x) \cup Z(F'/x) = X(F'/x)$, $M \models_{Y(F'/x)} \varphi$ and by locality, $M \models_{Z(F'/x)} \psi$. So $M \models_{X(F'/x)} \varphi \tai \psi$ and hence $M \models \on x(\varphi \tai \psi)$.
\item[(4')] Follows from locality of the semantics.
\item[(7)] As above it is enough to prove the equivalence for models $M$ and teams $X$ with $\Dom (X) =  \Fr (\phi) - \{x,y\}$. Assume first $M \models_X \forall x \on y (x \bot_{z} y \ja \phi)$. Then there is $F: X(M/x) \rightarrow \mathcal{P}(M)\setminus\{\emptyset\}$ such that if $X' = X(M/x)(F/y)$, then  $M \models_{X'} x \bot_{z} y \ja \phi$. If now $b \in M$ is such that there are $a \in M$ and $s \in X$ with $s(a/x)(b/y) \in X'$, then the independence atom guarantees that $s(a/x)(b/y) \in X'$ for all $a \in M$. Therefore, if we define $F': X \rightarrow \mathcal{P}(M)\setminus\{\emptyset\}$ so that 
$$F'(s) = \{b \in M \mid s(a/x)(b/y) \in X'\textrm{ for some }a \in M\},$$
then $X(F'/y)(M/x) = X(M/x)(F/y)$. Hence $M \models_X \on y \forall x \phi$.

For the converse, assume that $M \models_X \on y \forall x \phi$. Then there is $F: X \rightarrow \mathcal{P}(M)\setminus\{\emptyset\}$ such that if $X' = X(F/y)(M/x)$, then $M \models_{X'} \phi$. Clearly $M \models_{X'} x \bot_{z} y$ holds also. If we define $F' : X(M/x) \rightarrow \mathcal{P}(M)\setminus\{\emptyset\}$ so that $F'(s(a/x)) = F(s)$ for all $s \in X$ and $a \in M$, then $X(M/x)(F'/y) = X(F/y)(M/x)$. Hence $M \models_X \forall x \on y( x \bot_{z} y \ja \phi)$.
\end{enumerate}
\end{proof}

\begin{esim}
Generally it is not true that $M \models_X \forall x (\varphi \tai \psi) \Leftrightarrow M \models_X \forall x \varphi \tai \psi$ if $x$ does not occur free in  $\psi$. Let $\varphi := x \subseteq y \ja (x=1 \tai y=1)$ and $\psi := y = 0$. If $M$ is a model with domain $\{0,1\}$ and $X = \{\{(y,0)\},\{(y,1)\}\}$, then $M \models_X \forall x (\varphi \tai \psi)$ but $M \nvDash_X \forall x \varphi \tai \psi$. On the other hand, we can now see that $M \nvDash_X \forall x ((\varphi \ja x \bot y) \tai \psi)$.
\end{esim}

\begin{lem}[Rule 6]
Let
\begin{equation}
\phi_0 = \on x_0 (\bigwedge_{1 \leq i \leq m} u_i \bot_{w_i} v_i \ja \theta_0)
\end{equation}
and
\begin{equation}
\phi_1 = \on x_1 (\bigwedge_{m+1 \leq i \leq m + n} u_i \bot_{w_i} v_i \ja \theta_1)
\end{equation}
be formulas where $x_0$ is a tuple of variables that do not occur in $\phi_1$; $x_1$ is a tuple of variables that do not occur in $\phi_0$; $u_i$, $v_i$ and $w_i$ are tuples of bound variables; $\theta_0$ and $\theta_1$ are first-order formulas. Let $\a$, $\b$, $z_0$, $z_1$ and $r$ be variables that do not appear in formula $\phi_0 \tai \phi_1$. Then if we define
\begin{align*}
\varphi  = &\forall \a \forall \b \on x_0\on x_1 \on z_0\on z_1\on r [\bigwedge_{1 \leq i \leq m+n}  u_i \bot_{w_i r} v_i \ja \bigwedge_{i = 0,1}  \dep(z_i) \ja\\
&(\neg  z_0 = z_1 \tai \a = \b) \ja ((\theta_0  \ja r = z_0) \tai (\theta_1 \ja r = z_1))],
\end{align*}
we have that $\phi_0 \tai \phi_1 \equiv \varphi$.
\end{lem}

\begin{proof}
We divide the proof into two parts. First we prove that the equivalence holds for models $M$ with $|M|=1$ and then for models with larger domain. By locality of the semantics, we can without loss of generality assume that $X$ is always a team with $\Dom (X) = \Fr (\phi_0 \tai \phi_1)$. For notational simplicity we can without loss of generality assume that $x_0$ and $x_1$ are both of same length $l$.
\begin{enumerate}
\item Suppose $M$ is a model $|M| =1$ and $X$ is a team. If $M \models_X \phi_0 \tai \phi_1$, then $M \models_X \phi_0$ or $M \models_X \phi_1$. Now if we evaluate all the quantified variables in $\varphi$ by the only possible way, we have that $(\theta_0 \ja r=z_0)$ or $(\theta_1 \ja r=z_1)$ holds in $X$. Also $\a = \b$ must be true, so $(\neg z_0 = z_1 \tai \a = \b)$ holds in $X$. All the independence atoms are trivially true, so $M \models_X \varphi$.

Suppose then $M \models_X \varphi$. Then $X$ extended with values for $x_0$, $x_1$ must have $\theta_0$ or $\theta_1$ true. In either case independence atoms hold trivially, so $M \models_X \phi_0$ or $M \models_X \phi_1$. Hence $M \models_X \phi_0 \tai \phi_1$.

\item Suppose now $M$ is a model with $|M| >1$ and $X$ is a team. Let $0$ and $1$ be some distinct members of $M$.

Assume first that $M \models_X \phi_0 \tai \phi_1$. Then there are $Y,Z \subseteq X$, $Y \cup Z = X$, such that $M \models_Y \phi_0$ and $M \models_Z \phi_1$. Let $F_Y : Y \rightarrow \mathcal{P}(M^l)\setminus\{\emptyset\}$ and $F_Z : Z \rightarrow \mathcal{P}(M^l)\setminus\{\emptyset\}$ be functions witnessing this. Now we want to form a function $F : X(M^{2}/\a\b) \rightarrow \mathcal{P}(M^{2l+3})\setminus\{\emptyset\}$ so that if $X' = X(M^{2}/\a\b)(F/x_0x_1z_0z_1r)$, then $M$ and $X'$ satisfy the quantifier-free part of $\varphi$. First we define sets of tuples as follows:\\
Let $s \in X(M^{2}/\a\b)$. Define 
\begin{equation*}
\begin{split}
A_{s,z_0} = \{0\}\textrm{ and }A_{s,z_1} = \{1\}
\end{split}
\end{equation*}
and let
\begin{equation*}
\begin{split}
B_{s,x_0}& = F_Y(s\upharpoonright \Dom (X))\textrm{, }
B_{s,x_1}=\{0^l\}\textrm{ and }B_{s,r} =\{0\}\textrm{ if }s \upharpoonright \Dom (X) \in Y,\\ B_{s,x_0}& =B_{s,x_1}=B_{s,r}= \emptyset\textrm{ otherwise.}\\
C_{s,x_0}& = \{0^l\}\textrm{, }C_{s,x_1} = F_Z (s\upharpoonright \Dom (X))\textrm{ and }C_{s,r} = \{1\}\textrm{ if }s \upharpoonright \Dom (X) \in Z,\\ 
C_{s,y_0}& =C_{s,y_1}=C_{s,r}=\emptyset \textrm{ otherwise.}
\end{split}
\end{equation*}
Then define
\begin{equation*}
\begin{split}
B_s& =  {B_{s,x_0}}^\frown {B_{s,x_1}} ^\frown {A_{s,z_0}} ^\frown {A_{s,z_1}} ^\frown {B_{s,r}}\textrm{ and }\\
C_s& = {C_{s,x_0}}^\frown {C_{s,x_1}}^\frown {A_{s,z_0}} ^\frown {A_{s,z_1}} ^\frown {C_{s,r}}
\end{split}
\end{equation*}
and let $F(s) = B_s \cup C_s.$ Note that by the definition, $F(s)$ is non-empty for all $s \in X(M^2/\a \b)$.

Now it is enough to show that the quantifier-free part of $\varphi$ holds for $M$ and $X'$. So let us go through it part by part:
\begin{itemize}
\item $\bigwedge_{1 \leq i \leq m+n}  u_i \bot_{w_i r} v_i$: Let $i \leq m+n$ and $t,t' \in X'$  be such that $t(w_ir) = t'(w_ir)$. If they both evaluate $r$ as, say $0$, then by the definition of $F$, $t \upharpoonright  (\Dom (X) \cup \Var (x_0)),t'\upharpoonright (\Dom (X)\cup \Var(x_0)) \in Y(F_Y/x_0)$. If $i \leq m$, then this team satisfies $u_i \bot_{w_i} v_i$, and there is an assignment in $Y(F_Y/x_0)$ agreeing with $t$ for $u_iw_i$ and with $t'$ for $v_i$. Now we can extend it to an assignment $t''$ of $X'$ such that $t''(r)=0$. Then this $t''$ is as wanted. Suppose $i > m$. Then all the variables in tuples $u_i$, $v_i$ and $w_i$ are from tuple $x_1$ and $t(x_1)=t'(x_1)=0^l$. Thus we can choose $t''=t$.

The case where $t(r)=t'(r)=1$ is analogous.

\item $\bigwedge_{i = 0,1} \dep(z_i)$: Follows from the definition of $F$.

\item $\neg z_0 = z_1 \tai \a = \b$: Clearly $M \models_{X'} \neg z_0 = z_1$.

\item $(\theta_0  \ja r = z_0) \tai (\theta_1 \ja r = z_1)$: Simply divide $X'$ to $Y'$ and $Z'$ so that in $Y'$, $r=0$ and in $Z'$, $r=1$. Then $Y' \upharpoonright (\Dom (X) \cup \Var (x_0)) = Y(F_Y/x_0)$, so $\theta_0$ holds in $Y'$. Also $r=z_0$ holds trivially and hence $M \models_{Y'} \theta_0  \ja r = z_0$. Similarly $M \models_{Z'} \theta_1 \ja r = z_1$.
\end{itemize}

Assume then that $M \models_X \varphi$ and let $F : X(M^{2}/\a\b) \rightarrow \mathcal{P}(M^{2l+3})\setminus\{\emptyset\}$ be a function witnessing this. Then if $X'= X(M^{2}/\a\b)(F/x_0x_1z_0z_1r)$ we have that the quantifier-free part of $\varphi$ is true for $M$ and $X'$. Now define
\begin{align*}
Y&=\{s \in X \mid \on t \in X'[t \upharpoonright \Dom (X) = s \textrm{ and }t(r)=t(z_0)]\} \textrm{ and }\\
Z&=\{s \in X \mid \on t \in X'[t \upharpoonright \Dom (X) = s \textrm{ and }t(r)=t(z_1)]\}.
\end{align*}
Note that $M \models_t r=z_0 \tai r=z_1$ for all $t \in X'$, so $Y \cup Z = X$. Define also functions $F_Y : Y \rightarrow \mathcal{P}(M^l)\setminus\{\emptyset\}$ and $F_Z: Z \rightarrow \mathcal{P}(M^l)\setminus\{\emptyset\}$ by
\begin{align*}
F_Y(s)&=\{t(x_0)\mid t \in X'\textrm{, }t \upharpoonright \Dom (X) = s \textrm{ and }t(r)=t(z_0)\}\textrm{ and }\\ 
F_z(s)&=\{t(x_1)\mid t \in X'\textrm{, } t \upharpoonright \Dom (X) = s \textrm{ and }t(r)=t(z_1)\}.
\end{align*}
It is enough to show that 
\begin{equation}\label{eka}
M \models_{Y(F_Y/x_0)} \bigwedge_{1 \leq i \leq m} u_i \bot_{w_i} v_i \ja \theta_0
\end{equation}
and 
\begin{equation}\label{toinen}
M \models_{Z(F_Z/x_1)} \bigwedge_{m+1 \leq i \leq m+n} u_i \bot_{w_i} v_i \ja \theta_1.
\end{equation} 
For (\ref{eka}) assume first that $1 \leq i \leq m$ and $s,s' \in Y(F_Y/x_0)$ are such that $s(w_i)=s'(w_i)$. By the definition of $F_Y$, these assignments are extended by some $t,t'\in X'$ such that $t(r) = t(z_0)$ and $t'(r)=t'(z_0)$. Atom $\dep(z_0)$ holds in $X'$, so $t(r)=t'(r)$. Also $u_i \bot_{w_ir} v_i$ holds in $X'$, so there is $t'' \in X'$ such that $t''(u_iw_ir)=t(u_iw_ir)$ and $t''(v_i)=t'(v_i)$. Now also $t''(r)=t''(z_0)$, so $t''$ extends some $s'' \in Y(F_Y/x_0)$. Then $s''(u_iw_i)=t''(u_iw_i)=t(u_iw_i)=s(u_iw_i)$ and $s''(v_i)=t''(v_i)=t'(v_i)=s'(v_i)$, and hence $s''$ is as wanted. 

Then let us show that $M \models_{Y(F_Y/x)} \theta_0$. Consider this extension $t$ of $s$ such that $t(r)=t(z_0)$. First notice that $\a = \b$ cannot hold in whole $X'$ because $\a$ and $\b$ were universally quantified and $|M| > 1$. Therefore, for some assignment in $X'$, $\neg z_0 = z_1$ holds. But in $X'$ $z_0$ and $z_1$ are constants, so $\neg z_0 = z_1$ holds in whole $X'$. Hence $t(r) \neq t(z_1)$, and so $t$ belongs to the part of $X'$ where $\theta_0 \ja r=z_0$ holds. Therefore $M \models_s \theta_0$, and because $\theta_0$ is first-order, we have by definition that $M \models_{Y(F_Y/x_0)} \theta_0$.

The proof of (\ref{toinen}) is analogous. Hence $M \models_X \phi_0 \tai \phi_1$.
\end{enumerate}
\end{proof}

Notice that in the previous lemma parameters $\a$ and $\b$ were needed only for the case $|M|=1$. If we forget these trivial models, rule \ref{rule6} can be simplified.

Before going to the soundness proof, we need the following lemma. Recall that when using the  notation $\phi(x_{i_1}/x_1)\ldots(x_{i_n}/x_n)$ we presume  that none of the variables $x_{i_1}, \ldots ,x_{i_n}$ become bound in the substitution.

\begin{lem}[Change of free variables]\label{freevar}
Let the free variables of $\phi$ be $x_1, \ldots ,x_n$. Let $i_1, \ldots ,i_n$ be distinct. If $X$ is a team with $\Dom (X)=\{x_1, \ldots ,x_n\}$, let $X'$ consist of the assignments $x_{i_j} \mapsto s(x_j)$ where $s \in X$. Then
\begin{equation*}
M \models_X \phi \Leftrightarrow M \models_{X'} \phi(x_{i_1}/x_1)\ldots(x_{i_n}/x_n).
\end{equation*}
\end{lem}
\begin{proof}
Easy induction on the complexity of the formula.
\end{proof}

\begin{prop}\label{soundnessthm} Let $T \cup \{\psi\}$ be a set of formulas of independence logic. If $T \vdash_{\I} \psi$, then $T \models \psi$.
\end{prop}

\begin{proof} We will prove this claim by induction on the length of derivation. First notice that the previous lemmas provide the soundness of rules 3, 4, 6 and 7. Rules 2, 11, 12, 13, 14, $\ja$ I, $\ja$ E, $\tai$ I and $\neg$ E are obviously sound. Also rules $\forall$ I, $\on$ I and $\on$ E are identical to the corresponding rules in the dependence logic case and  the proof for these is as in \cite{juhajouko}. Note that these rules do not apply downward closure. Rule $\forall$ E is a restricted version of the corresponding dependence logic rule and also here the proof introduced in \cite{juhajouko} suffices. The rules presented in Figure \ref{dod} form the well-known sound and complete axiomatization for inclusion dependencies (see \cite{cfp}), and are clearly sound in this context too. Hence, we  prove induction steps for rules \ref{rule1}, \ref{rule5}, \ref{rule8}, \ref{rule9} and \ref{rule10}. Note that in the soundness proof of rule 1, we use the fact that all the non-discharged assumptions used in the derivation of $C$ are downwards closed.  The soundness proofs of $\tai$ E and $\neg$ I are then analogous applications of this principle, and hence omitted.
\begin{itemize}
\item[Rule 1] Assume that we have a natural deduction proof of $A\tai C$ from the  assumptions 
\begin{equation*}
 \{A_1,\ldots , A_k\}
\end{equation*}
with last rule 1. Let $M$ and $X$ be such that  $M\models _X A_i$ for $i=1, \ldots ,k$. By the assumption, we have a shorter proof of $A\tai B$ from the assumptions $ \{A_1,\ldots, A_k\}$. Then by the induction assumption, $M \models_X  A\tai B$, and hence there exist $Y,Z\subseteq X$, $Y\cup Z = X$, such that $M \models_Y A$ and $M\models_Z B$. By the assumption we have also a shorter proof of $C$ from $\{B,A_{i_1}, \ldots,A_{i_l}\}$ where $\{A_{i_1}, \ldots ,A_{i_l}\} \subseteq \{A_1, \ldots ,A_k\}$ is a set of first-order formulas. Now by Proposition \ref{flat}, $M \models_Z A_{i_j}$ for $j=1, \ldots ,l$. Therefore  by the induction assumption  $M\models_Z C$, and hence we conclude that $M\models _X A\tai C$.

\item[Rule 5] Assume that we have a natural deduction proof of $\forall y B$ from the  assumptions 
\begin{equation*}
 \{A_1,\ldots , A_k\}
\end{equation*}
with last rule 5. Let $M$ and $X$ be such that  $M\models _X A_i$ for $i=1, \ldots ,k$. By the assumption, we have a shorter proof of $\forall x A$ from the assumptions $ \{A_1,\ldots, A_k\}$. Then by the induction assumption, $M \models_X \forall x A$. Let $V = \Dom (X)-\{x,y\}$ and $X'=X \upharpoonright V$. Variables $x$ and $y$ do not occur free in $\forall x A$, so also $M \models_{X'} \forall x A$ and hence $M \models_{X'(M/x)} A$. By Lemma \ref{freevar}, $M \models_{X'(M/y)} A(y/x)$. Because $X'(M/y)=X(M/y) \upharpoonright (V \cup \{y\})$ and $x$ does not occur free in $A(y/x)$, we have that $M \models_{X(M/y)} A(y/x)$. Also by the assumption, we have a shorter proof of $B$ from the assumptions
\begin{equation*}
 \{A(y/x),A_{i_1},\ldots , A_{i_l}\}
\end{equation*}
where $\{A_{i_1},\ldots,A_{i_l}\} \subseteq \{A_1, \ldots ,A_k\}$ and $y$ does not occur free in $A_{i_j}$ for $j=1, \ldots ,l$. Hence $M \models_{X(M/y)} A_{i_j}$ for $j=1, \ldots ,l$, so by the induction assumption, $M \models_{X(M/y)} B$. Hence $M \models_{X} \forall y B$.
\item[Rule \ref{rule8}] Assume that we have a natural deduction proof of $A(y/x)$ from the  assumptions $ \{A_1,\ldots , A_k\}$ with last rule \ref{rule8}. Let $M$ and $X$ be such that  $M\models _X A_i$ for $i=1, \ldots ,k$. By the assumption, we have shorter proofs of $y \sub x$ and $A$ from the assumptions $ \{A_1,\ldots, A_k\}$. Then by the induction assumption, 
\begin{equation}\label{81}
M \models_X y \sub x
\end{equation}
 and $M \models_X A$. Assume that  $x=(x_1, \ldots ,x_n)$ and $y=(y_1, \ldots ,y_n)$, and let $V_x:= \{x_1, \ldots ,x_n\}$ and $V_y:= \{y_1, \ldots ,y_n\}$. Since $V_x = \Fr(A)$, we first obtain by Proposition \ref{loc} that $M \models_{X \upharpoonright V_x} A$. Then letting $X'$ consist of the assignments $y_i \mapsto s(x_i)$, for $s \in X \upharpoonright V$, we obtain by Lemma \ref{freevar} that
$M \models_{X'} A(y/x)$. Hence, and since $X \upharpoonright V_y \sub X'$ by \eqref{81}, we obtain by Proposition \ref{flat} that $M \models_{X\upharpoonright V_y} A(y/x)$. For this, note that $A(y/x)$ is first-order by the prerequisite. Therefore, by Proposition \ref{loc} $M \models_{X} A(y/x)$.

\item[Rule \ref{rule9}]
Assume that we have a natural deduction proof of 
$
\on x_2( x_2 \sub x \ja (A \tai u_0v_1w_0 = u_2v_2w_2))
$
from the assumptions $ \{A_1,\ldots , A_k\}$ with last rule \ref{rule9}. Let $M$ and $X$ be such that  $M\models _X A_i$ for $i=1, \ldots ,k$. By the assumption, we have shorter proofs of $\indep{w}{u}{v} $, $x_0 \sub x $ and  $x_1 \sub x $ from the assumptions $ \{A_1,\ldots, A_k\}$. Then by the induction assumption,\begin{itemize}
\item[$(i)$] $M \models_X \indep{w}{u}{v}$, 
\item[$(ii)$] $M \models_X x_0 \sub x $,
\item[$(iii)$] $M \models_X x_1 \sub x $.
\end{itemize}
It suffices to define a $F: X \rightarrow  \mathcal{P}(M^{|x_2|}) \setminus \{\emptyset\}$ such that 
\begin{equation}\label{91}
M \models_{X(F/x_2)}x_2 \sub x \ja (A \tai u_0v_1w_0 = u_2v_2w_2).
\end{equation} Let $s \in X$. By $(ii)$ and $(iii)$, there exist $s',s''\in X$ such that $s(x_0)=s'(x)$ and $s(x_1) = s''(x)$. If $s'(w) \neq s''(w)$, then we let $F(s)=\{s(x)\}$. If $s'(w) = s''(w)$, then by $(i)$ we can choose a $s^* \in X$ such that $s^*(u)s^*(v)s^*(w)=s'(u)s''(v)s'(w)$, and let $F(s)=\{s^*(x)\}$. Recall that in \eqref{91}, $A$ is $\perp$ if $w$ is empty, and otherwise $A$ is $\neg w_0 = w_1$. Therefore, it is straightforward to show that with this definition of $F$, \eqref{91} follows. 

\item[Rule \ref{rule10}]
Assume that we have a natural deduction proof of 
$$\forall x \on y \big ( B\ja 
\forall u \on v \on x' \on y' (A(uv/xy) \ja x'y'=xy \ja x'y' \sub uv)\big ) 
$$
from the assumptions $ \{A_1,\ldots , A_k\}$ with last rule \ref{rule10}. Let $M$ and $X$ be such that  $M\models _X A_i$ for $i=1, \ldots ,k$. By the assumption, we have a shorter proof of $\forall x \on y (A \ja B)$ from the assumptions $ \{A_1,\ldots, A_k\}$. Then by the induction assumption there exists a $F: X(M^{|x|}/x)\rightarrow   \mathcal{P}(M^{|y|}) \setminus \{\emptyset\}$ such that $M \models_{X'} A\ja B$, for $X' := X(M^{|x|}/x)(F/y)$. It suffices to define two functions $G_0: X'(M^{|u|}/u) \rightarrow \mathcal{P}(M^{|v|}) \setminus \{\emptyset\}$ and $G_1 :X'(M^{|u|}/u)(G_0/v) \rightarrow \mathcal{P}(M^{|x'y'|}) \setminus \{\emptyset\}$ such that
$$M \models_{X''} A(uv/xy) \ja x'y'=xy \ja x'y' \sub uv$$
where $X'':= X'(M^{|u|}/u)(G_0/v)(G_1/x'y')$. We define
\begin{itemize}
\item $G_0(s) = \{s'(y) \mid  s' \in X', s'(x)=s(u)\}$, for $s \in X'(M^{|u|}/u) $,
\item $G_1(s) = \{s(xy)\}$, for $s \in X'(M^{|u|}/u)(G_0/v)$.
\end{itemize}
Now, if we define $V$ as the set of variables listed in $uv$, then $X'' \upharpoonright V$ consists of the assignments $uv \mapsto s(xy)$, for $s \in X'$. Therefore, since $xy$ lists $\Fr(A)$ and $M \models_{X'} A$, using Proposition \ref{loc} and Lemma \ref{freevar}, we obtain that $\M \models_{X''} A(uv/xy)$. Also by the construction, $\M \models_{X''} x'y'=xy \ja x'y' \sub uv$. This concludes the proof.

\end{itemize}
\end{proof}

\section{The Completeness Theorem}\label{complete}
In this section we will show that using our system of natural deduction we can derive all the first-order consequences of sentences of independence logic. Our proof is analogous to the proof of the corresponding dependence logic theorem in Kontinen and V\"a\"an\"anen \cite{juhajouko} which in turn builds on the earlier work of Barwise \cite{jon} by using first-order approximations in the completeness proof.
\subsection{The roadmap for the proof}
\begin{enumerate}
\item First we will show that from any independence logic sentence $\phi$ it is possible to derive an equivalent sentence of the form
\begin{equation}\label{norm}
\phi' = \forall x \on y (\bigwedge_{1 \leq i \leq m} u_i \bot_{w_i} v_i \ja \theta)
\end{equation}
where $x$ and $y$ are tuples of variables where each variable is quantified only once; $u_i$, $v_i$ and $w_i$ are tuples of existentially quantified variables and $\theta$ is a quantifier-free first-order formula.

\item The sentence $\phi'$ can be shown to be equivalent, in countable models, to the game expression
\begin{align*}
\Phi:=&\forall x_{0,0} \on y_{0,0} (\Psi^0 \ja \\
\forall x_{1,-1} \on y_{1,-1} &\on x_{1,0} \on y_{1,0} \ldots \on x_{1,p_1} \on y_{1,p_1} (\Psi^1 \ja \\
\forall  x_{2,-2} \on y_{2,-2} \on x_{2,-1} \on y_{2,-1} &\on x_{2,0} \on y_{2,0}\ldots \qquad\ldots  \on x_{2,p_2} \on y_{2,p_2} (\Psi^2 \ja \\
&\qquad\ldots \\
&\qquad\ldots \\
&\hspace{1cm} \ldots ))).
\end{align*}

In the game expression, $\Psi^0 := \theta_{0,0}$, and for $n \geq 1$,
\begin{align*}
\Psi^n :=&\bigwedge_{-n \leq i \leq p_n} \theta_{n,i} \ja \bigwedge_{-n+1 \leq i \leq p_{n-1}} x_{n,i} y_{n,i} = x_{n-1,i} y_{n-1,i} \ja \\
 &\bigwedge_{\substack{1 \leq i \leq m\\ -n \leq j,k \leq p_{n-1}}} (\pi^i_{n,j,k} \tai
 \bigvee_{p_{n-1} < l \leq p_n} u^i_{n,j} v^i_{n,k} w^i_{n,j} = u^i_{n,l} v^i_{n,l} w^i_{n,l}))
\end{align*}
where 
\begin{itemize}
\item $x_{j,k}$ and $x$ are tuples of same length and $y_{j,k}$ and $y$ are tuples of same length such that each variable in these tuples is quantified only once,
\item $\theta_{j,k}=\theta(x_{j,k}y_{j,k}/xy)$,
\item $e^i_{n,j} = e_i(x_{n,j}y_{n,j}/xy)$ for $e \in \{u,v,w\}$,
\item $ \pi^i_{n,j,k} =  \left\{\begin{array}{l l}
    \perp & \quad \textrm{if $w^i$ is empty}\\
    \neg w^i_{n,j}=w^i_{n,k} & \quad \textrm{otherwise}\\
  \end{array}\right.$,
\item $p_0 =0$ and $p_{n} = p_{n-1} + m(p_{n-1}+n+1)^2$, for $n \geq 1$.
\end{itemize}

The idea behind the game expression is that at level $n$, $x_{n,-n}y_{n,-n}$ introduces a new tuple of $M$, tuple $x_{n,i}y_{n,i}$, for $i=-n+1, \ldots ,p_{n-1}$, copies all the tuples introduced at the previous level and tuple $x_{n,i}y_{n,i}$, for $i=p_{n-1}+1, \ldots ,p_n$, confirms that the independence atoms hold between all the tuples $x_{n,i}y_{n,i}$ and $x_{n,j}y_{n,j}$ where $-n \leq i,j \leq p_{n-1}$.

\item The game expression $\Phi$ can be approximated by the first-order formulas
\begin{align*}
\Phi_n:=&\forall x_{0,0} \on y_{0,0} (\Psi^0 \ja \\
\forall x_{1,-1} \on y_{1,-1} &\on x_{1,0} \on y_{1,0} \ldots \on x_{1,p_1} \on y_{1,p_1} (\Psi^1 \ja \\
\forall  x_{2,-2} \on y_{2,-2} \on x_{2,-1} \on y_{2,-1} &\on x_{2,0} \on y_{2,0}\ldots \qquad\ldots  \on x_{2,p_2} \on y_{2,p_2} (\Psi^2 \ja \\
&\hspace{1cm}\ldots \\
&\hspace{1cm}\ldots \\
\forall x_{n,-n} \on y_{n,-n} \on x_{n,-n+1} \on y_{n,-n+1} \ldots\quad \ldots  &\on x_{n,0}\on y_{n,0} \ldots\hspace{1.2cm} \ldots   \on x_{n,p_n} \on y_{n,p_n} (\Psi^n) \ldots ))).
\end{align*}

\item Then we will show that these approximations can all be deduced from $\phi'$.

\item Then we note that for recursively saturated (or finite) models $M$, it holds that
$$M \models \Phi \leftrightarrow \bigwedge_{n} \Phi^n.$$

\item At last we show that for any $T \subseteq \I$ and $\phi \in$ FO:
$$T \models \phi \Leftrightarrow T \vdash \phi.$$
Suppose $T\nvdash \phi$. If $T^*$ consist of the first-order approximations of sentences of $T$, then $T^* \nvdash \phi$ and $T^*\cup \{\neg \phi\}$ is deductively consistent in first-order logic. Taking some countable recursively saturated model of $T^*\cup \{\neg \phi\}$, we have a model of $T \cup \{\neg \phi\}$ and hence $T \not\models \phi$.
\end{enumerate}

\subsection{From $\phi$ to $\phi'$}

In this section we are going to prove that from $\phi$ one can derive an equivalent formula $\phi'$ of the form
\begin{equation}\label{norm}
\forall x \on y (\bigwedge_{1 \leq i \leq m} u_i \bot_{w_i} v_i \ja \theta)
\end{equation}
where $x$ and $y$ are tuples of variables where each variable is quantified only once; $u_i$, $v_i$ and $w_i$ are tuples of existentially quantified variables and $\theta$ is a quantifier-free first-order formula.

\begin{prop}
Let $\phi$ be a sentence of independence logic. Then $\phi \vdash_{\I} \phi'$ where $\phi$ and $\phi'$ are logically equivalent and $\phi'$ is of the form (\ref{norm}).
\end{prop}
\begin{proof}
We will prove the claim in several steps. Without loss of generality we may assume that in $\phi$ each variable is quantified only once.

\begin{itemize}
\item[Step 1] We derive from $\phi$ an equivalent sentence in prenex normal form
\begin{equation}\label{Q}
Q^1x_{i_1} \ldots Q^nx_{i_n}\theta
\end{equation}
where $Q^i\in \{\on ,\forall\}$ and $\theta$ is a quantifier-free formula.

We will prove this for every formula $\phi$ satisfying the assumption made in the beginning of the proof and the assumption that no variable appears both free (if $\phi$ has free variables) and bound in the formula. Now if $\phi$ is atomic or first-order formula, then the claim clearly holds. (In the latter case we know that our deduction system covers the natural first-order deduction system and in that system we can derive an equivalent formula in prenex normal form.) Also the cases of universal and existential quantifications are trivial. So we need only to consider the cases of disjunction and conjunction. We prove these cases by simultaneous induction.

Assume $\phi = \psi \tai \theta$ or $\phi =\psi \ja \theta$. By the induction assumption, we have derivations $\psi \vdash_{\I} \psi^*$ and $\theta \vdash_{\I} \theta^*$ where
\begin{eqnarray*}
\psi^* &=& Q^1 x_{i_1} \ldots Q^n x_{i_n} \psi_0, \\
\theta^* &=& Q^{n+1} x_{i_{n+1}} \ldots Q^{n+m}x_{i_{n+m}}\theta_0,
\end{eqnarray*}
and $\psi \equiv \psi^*$ and $\theta \equiv \theta^*$. If $\phi = \psi \tai \theta$, we can derive $\psi^* \tai \theta^*$ from $\phi$ using applications of rules 1 and 2. If $\phi = \psi \ja \theta$, we can derive $\psi^* \ja \theta^*$ from $\phi$ using applications of rules $\ja$ I and $\ja$ E. Next we prove by induction on $n$ that from $\psi^* \ja \theta^*$ we can derive an equivalent formula
\begin{equation}
Q^1x_{i_1} \ldots Q^nx_{i_n} Q^{n+1}x_{i_{n+1}} \ldots Q^{n+m}x_{i_{n+m}} (\psi_0 \ja \theta_0)
\end{equation}
and from $\psi^* \tai \theta^*$ we can derive an equivalent formula
\begin{equation}
Q^1x_{i_1} \ldots Q^nx_{i_n} Q^{n+1}x_{i_{n+1}} \ldots Q^{n+m}x_{i_{n+m}} (\psi_1 \tai \theta_1)
\end{equation}
where $\psi_1$ and $\theta_1$ are quantifier-free formulas. Let $n=0$. We prove this case also by induction, this time on $m$. For $m=0$ the claim holds. Suppose $m = k+1$ and the claim holds for $k$. We consider only the case where the connective is $\tai$ and $Q^1 = \forall$. The other cases are analogous, except that they are a bit easier. The following deduction shows the claim:
\begin{enumerate}
\item $\psi_0 \tai Q^1x_{i_1} \ldots Q^mx_{i_m} \theta_0$ 

\item $Q^1x_{i_1} \ldots Q^mx_{i_m} \theta_0 \tai \psi_0$ (rule 2)

\item $Q^1x_{i_1} ((Q^2x_{i_2} \ldots Q^mx_{i_m} \theta_0 \ja x_{i_1} \bot y) \tai \psi_0)$ (rule 3)

\item $Q^1x_{i_1} Q^2x_{i_2} \ldots Q^mx_{i_m} (\psi_1\tai\theta_1 )$ (rule 5 and D1)
\end{enumerate}
where D1 is the derivation
\begin{enumerate}
\item $(Q^2x_{i_2} \ldots Q^mx_{i_m} \theta_0 \ja x_{i_1} \bot y) \tai \psi_0$

\item $\psi_0 \tai (Q^2x_{i_2} \ldots Q^mx_{i_m} \theta_0 \ja x_{i_1} \bot y)$ (rule 2)

\item $\psi_0 \tai Q^2x_{i_2} \ldots Q^mx_{i_m} ( x_{i_1} \bot y\ja\theta_0)$ (rule 1\footnote{Rule 1 can be applied since no extra assumptions are used in D2. In the sequel we  apply rule 1 analogously.} and D2)

\item \hspace{1cm} .

\item  \hspace{1cm} .

\item  \hspace{1cm} .

\item $Q^2x_{i_2} \ldots Q^mx_{i_m}  (\psi_1\tai\theta_1)$ (induction assumption)
\end{enumerate}
where D2 is the derivation
\begin{enumerate}
\item $Q^2x_{i_2} \ldots Q^mx_{i_m} \theta_0 \ja x_{i_1} \bot y$

\item $x_{i_1} \bot y \ja Q^2x_{i_2} \ldots Q^mx_{i_m} \theta_0$ ($\ja$ E and $\ja$ I)

\item  \hspace{1cm} .

\item \hspace{1cm} .

\item\hspace{1cm} .

\item $Q^2x_{i_2} \ldots Q^mx_{i_m} (x_{i_1} \bot y\ja\theta_0)$ (induction assumption)
\end{enumerate}
We can use the induction assumption in the deduction because $x_{i_j}$ are all different from each other and none of them are in tuple $y$. This concludes the proof for the case $n=0$.

Assume then that $n=l+1$ and that the claim holds for $l$. We show the claim in the case where the connective is $\tai$ and $Q^1=\forall$. The other cases are again analogous.

\begin{enumerate}
\item $Q^1x_{i_1}\ldots Q^nx_{i_n} \psi_0 \tai Q^{n+1}x_{i_{n+1}} \ldots Q^{n+m}x_{i_{n+m}} \theta_0$

\item $Q^1x_{i_1}((Q^2x_{i_2} \ldots Q^nx_{i_n} \psi_0 \ja x_{i_1} \bot y) \tai Q^{n+1}x_{i_{n+1}} \ldots Q^{n+m}x_{i_{n+m}} \theta_0)$ (rule 3)

\item $Q^1x_{i_1} \ldots Q^{n+m}x_{i_{n+m}} (\psi_1 \tai \theta_1)$ (rule 5 and D3)
\end{enumerate}

where D3 is the derivation
\begin{enumerate}
\item $(Q^2x_{i_2} \ldots Q^nx_{i_n} \psi_0 \ja x_{i_1} \bot y) \tai Q^{n+1}x_{i_{n+1}} \ldots Q^{n+m}x_{i_{n+m}} \theta_0$

\item $Q^{n+1}x_{i_{n+1}} \ldots Q^{n+m}x_{i_{n+m}} \theta_0 \tai (Q^2x_{i_2} \ldots Q^nx_{i_n} \psi_0 \ja x_{i_1} \bot y)$ (rule 2)

\item $Q^{n+1}x_{i_{n+1}} \ldots Q^{n+m}x_{i_{n+m}} \theta_0 \tai Q^2x_{i_2} \ldots Q^nx_{i_n} (\psi_0 \ja x_{i_1} \bot y)$ (rule 1 and D4)

\item $Q^2x_{i_2} \ldots Q^nx_{i_n} (\psi_0 \ja x_{i_1} \bot y) \tai Q^{n+1}x_{i_{n+1}} \ldots Q^{n+m}x_{i_{n+m}} \theta_0$ (rule 2)

\item \hspace{1cm} .

\item  \hspace{1cm} .

\item \hspace{1cm} .

\item $Q^2x_{i_2} \ldots Q^{n+m}x_{i_{n+m}} (\psi_1 \tai \theta_1)$ (induction assumption)
\end{enumerate}

where D4 is the derivation
\begin{enumerate}
\item $Q^2x_{i_2} \ldots Q^nx_{i_n} \psi_0 \ja x_{i_1} \bot y$

\item  \hspace{1cm} .

\item \hspace{1cm} .

\item \hspace{1cm} .

\item $Q^2x_{i_2} \ldots Q^nx_{i_n} (\psi_0 \ja x_{i_1} \bot y)$ (induction assumption)
\end{enumerate}
This concludes the proof.
\item[Step 2] Next we show that from a quantifier-free formula $\theta$ one can derive an equivalent formula of the form
\begin{equation}\label{theta*}
\forall y_1 \ldots \forall y_{l} \on y_{l+1} \ldots \on y_{l+l'}(\bigwedge_{1 \leq i \leq m} u_i \bot_{w_i} v_i \ja \theta^*)
\end{equation}
where $\theta^*$ is a quantifier-free first-order formula and $u_i$, $v_i$ and $w_i$ are tuples of existentially quantified variables. We do this by induction on the complexity of the formula. If $\theta$ is first-order formula, then the claim holds. Assume that $\theta = t \bot_{t''} t'$ where $t$, $t$ and $t''$ are tuples of terms $(s_1, \ldots ,s_k)$, $(s_{k+1}, \ldots ,s_{k+k'})$ and $(s_{k+k'+1}, \ldots ,s_{k+k'+k''})$, respectively. Let $l=k+k'+k''$. Assume that $0 \leq n < l$ and we have already derived
\begin{equation}\label{ntimes}
\on y_1 \ldots \on y_{n} (t_{n} \bot_{t''_{n}} t'_{n}  \ja y_1 = s_1 \ja \ldots \ja y_{n} = s_{n})
\end{equation}
where $t_{i}$ refers to the tuple $t(y_1/s_1)\ldots(y_{i}/s_{i})$ and tuples $t'_{i}$ and $t''_{i}$ are defined analogously.

Let D5 be the derivation
\begin{enumerate}
\item $t_{n} \bot_{t''_n} t'_n \ja y_1 = s_1 \ja \ldots \ja y_{n} = s_{n}$

\item $t_{n} \bot_{t''_n} t'_n \ja y_1 = s_1 \ja \ldots \ja y_{n} = s_{n} \ja s_{n+1} = s_{n+1}$ (Here we obtain first,  by rule \ref{rule11},  a dummy $x=x$ from which we obtain  $s_{n+1} = s_{n+1}$ by rule \ref{rule13}. Then we apply $\ja$ I.)

\item $\on y_{n+1} (t_{n+1} \bot_{t''_{n+1}} t'_{n+1} \ja y_1 = s_1 \ja \ldots \ja y_{n} = s_{n} \ja y_{n+1} =s_{n+1})$ ($\on$ I)
\end{enumerate}

The last step can be done if we interpret the second formula as $\phi(s_{n+1}/y_{n+1})$ for
$$\phi = t_{n+1} \bot_{t''_{n+1}} t'_{n+1} \ja y_1 = s_1 \ja \ldots \ja y_{n} = s_{n} \ja y_{n+1} = s_{n+1}.$$

Using $n$ times rule $\on E$, once D5 and $n$ times rule $\on I$, we can derive
$$\on y_1 \ldots \on y_{n+1} (t_{n+1} \bot_{t''_{n+1}} t'_{n+1} \ja y_1 = t_1 \ja \ldots \ja y_{n+1} = s_{n+1})$$
from (\ref{ntimes}).

So from $\theta$ one can derive
\begin{equation}\label{atomistep}
\on y_1 \ldots \on y_{l} (t_{l} \bot_{t''_{l}} t'_{l}  \ja y_1 = s_1 \ja \ldots \ja y_{l} = s_{l})
\end{equation}
which is clearly equivalent to $\theta$ and of the required form.

Assume then that $\theta = \phi \tai \psi$. By the induction assumption, we have derivations $\phi \vdash_{\I} \phi^*$ and $\psi \vdash_{\I} \psi^*$ where
\begin{eqnarray}\label{phi^*}
\phi^* &=& \forall y_1 \on y_2  (\bigwedge_{1 \leq i \leq m_1} u_i \bot_{w_i} v_i \ja \phi_0),\\
\psi^* &=& \forall y'_1 \on y'_2 (\bigwedge_{1 \leq i \leq m_2} u'_i \bot_{w'_i} v'_i \ja \psi_0)\label{psi^*}
\end{eqnarray}
such that $\phi \equiv \phi^*$, $\psi \equiv \psi^*$, $\phi_0$ and $\psi_0$ are quantifier-free first-order formulas, $y_i$ and $y'_i$, for $i=1,2$, are tuples of bound variables such that none of these variables occur in both formulas or are quantified more than once, $e_i$ is a tuple of existentially quantified variables for $e \in \{u,v,w,u',v',w'\}$.

Now $\theta \vdash_{\I} \phi^* \tai \psi^*$. First we show by induction on the length of $y_1$ that from $\phi^* \tai \psi^*$ one can derive an equivalent formula of the form
\begin{eqnarray}\label{puoliksioikein}
\forall y_1 \forall y'_{1} (\on y_3 ( \bigwedge_{1 \leq i \leq m_3} u_i \bot_{w_i} v_i \ja \phi_1 )\tai
 \on y'_3 (\bigwedge_{1 \leq i \leq m_4} u'_i \bot_{w'_i} v'_i  \ja \psi_1))
\end{eqnarray}
where $\phi_1$ and $\psi_1$ are quantifier-free first-order formulas, $y_3$ and $y'_3$ are tuples of bound variables such that none of these variables are quantified more than once or occur free in the formula, $e_i$ is a tuple of existentially quantified variables for $e \in \{u,v,w,u',v',w'\}$.

Assume first that $\len (y_1)=0$. We show this case by induction on the length of $y'_1$. The case $\len (y'_1)=0$ is clear. Suppose $\len (y'_1) = k+1$. Let $y'_1=xy'_4$ where $\len (y'_4)=k$ and let $y$ be a tuple listing the free variables in $\phi^* \tai \psi^*$. The following deduction shows the claim.

\begin{enumerate}
\item $\on y_2 (\bigwedge_{1 \leq i \leq m_1} u_i \bot_{w_i} v_i \ja \phi_0) \tai
\forall y'_{1} \on y'_{2}(\bigwedge_{1 \leq i \leq m_2} u'_i \bot_{w'_i} v'_i \ja \psi_0)$

\item $\forall y'_{1} \on y'_2 (\bigwedge_{1 \leq i \leq m_2} u'_i \bot_{w'_i} v'_i \ja \psi_0)\tai
\on y_2  (\bigwedge_{1 \leq i \leq m_1} u_i \bot_{w_i} v_i \ja \phi_0)$ (rule 2)

\item $\forall x ((\forall y'_4 \on y'_2 (\bigwedge_{1 \leq i \leq m_2} u'_i \bot_{w'_i} v'_i \ja \psi_0) \ja x \bot y) \tai \on y_2  (\bigwedge_{1 \leq i \leq m_1} u_i \bot_{w_i} v_i \ja \phi_0))$ (rule 3)

\item $\forall y'_1 (\on y_3 (\bigwedge_{1 \leq i \leq m_3} u_i \bot_{w_i} v_i \ja \phi_1) \tai \on y'_3 (\bigwedge_{1 \leq i \leq m_4} u'_i \bot_{w'_i} v'_i \ja\psi_1))$ (rule 5 and D6)
\end{enumerate}

where D6 is the derivation

\begin{enumerate}
\item $(\forall y'_4 \on y'_2 (\bigwedge_{1 \leq i \leq m_2} u'_i \bot_{w'_i} v'_i \ja \psi_0) \ja x \bot y) \tai \on y_2  (\bigwedge_{1 \leq i \leq m_1} u_i \bot_{w_i} v_i \ja \phi_0)$

\item $\on y_2  (\bigwedge_{1 \leq i \leq m_1} u_i \bot_{w_i} v_i \ja \phi_0) \tai (\forall y'_4 \on y'_2 (\bigwedge_{1 \leq i \leq m_2} u'_i \bot_{w'_i} v'_i \ja \psi_0) \ja x \bot y)$ (rule 2)

\item $\on y_2  (\bigwedge_{1 \leq i \leq m_1} u_i \bot_{w_i} v_i \ja \phi_0) \tai \forall y'_4 \on y'_2 \on a \on b( \bigwedge_{1 \leq i \leq m_2} u'_i \bot_{w'_i} v'_i \ja a \bot b  \ja \psi_0 \ja ab=xy))$ (rule 1 and D7)

\item \hspace{1cm}.

\item \hspace{1cm}.

\item \hspace{1cm}.

\item $\forall y'_4 (\on y_3 (\bigwedge_{1 \leq i \leq m_3} u_i \bot_{w_i} v_i \ja \phi_1) \tai \on y'_3 (\bigwedge_{1 \leq i \leq m_4} u'_i \bot_{w'_i} v'_i \ja\psi_1))$ (induction assumption)
\end{enumerate}

where D7 is the derivation

\begin{enumerate}
\item $\forall y'_4 \on y'_2 (\bigwedge_{1 \leq i \leq m_2} u'_i \bot_{w'_i} v'_i \ja \psi_0) \ja x \bot y$

\item \hspace{1cm}.

\item \hspace{1cm}.

\item \hspace{1cm}.

\item $\forall y'_4 \on y'_2 (\bigwedge_{1 \leq i \leq m_2} u'_i \bot_{w'_i} v'_i \ja \psi_0) \ja \on a \on b(a \bot b \ja ab=xy)$ (Here we use $\ja$ E and $\ja$ I and deduce the second conjunct as we deduced (\ref{atomistep}) previously.)

\item \hspace{1cm} .

\item \hspace{1cm} .

\item \hspace{1cm} .

\item $\forall y'_4 \on y'_2 \on a \on b( \bigwedge_{1 \leq i \leq m_2} u'_i \bot_{w'_i} v'_i \ja a \bot b  \ja \psi_0 \ja ab=xy)$ (Using rules $\on$ I, $\on$ E, 5, $\ja$ I and $\ja$ E, we can drag the quantifiers to the left side of the formula and rearrange the quantifier-free part as we want.)
\end{enumerate}
This concludes the proof of this case.

Suppose then $\len (y_1)=n+1$. Let $y_1=xy_4$ where $\len (y_4)=n$ and let $y$ be a tuple listing the free variables in $\phi^* \tai \psi^*$. The following deduction shows the claim

\begin{enumerate}
\item $\forall y_1 \on y_2 (\bigwedge_{1 \leq i \leq m_1} u_i \bot_{w_i} v_i \ja \phi_0) \tai
\forall y'_{1} \on y'_{2}(\bigwedge_{1 \leq i \leq m_2} u'_i \bot_{w'_i} v'_i \ja \psi_0)$

\item $\forall x ((\forall y_4 \on y_2 (\bigwedge_{1 \leq i \leq m_1} u_i \bot_{w_i} v_i \ja \phi_0)\ja x \bot y) \tai
\forall y'_{1} \on y'_{2}(\bigwedge_{1 \leq i \leq m_2} u'_i \bot_{w'_i} v'_i \ja \psi_0))$ (rule 3)

\item $\forall y_1\forall y'_1(\on y_3 (\bigwedge_{1 \leq i \leq m_3} u_i \bot_{w_i} v_i \ja \phi_1) \tai \on y'_{3}(\bigwedge_{1 \leq i \leq m_4} u'_i \bot_{w'_i} v'_i \ja \psi_1)$ (rule 5 and D8)
\end{enumerate}

where D8 is the derivation

\begin{enumerate}
\item $(\forall y_4 \on y_2 (\bigwedge_{1 \leq i \leq m_1} u_i \bot_{w_i} v_i \ja \phi_0)\ja x \bot y) \tai
\forall y'_{1} \on y'_{2}(\bigwedge_{1 \leq i \leq m_2} u'_i \bot_{w'_i} v'_i \ja \psi_0)$

\item $\forall y'_{1} \on y'_{2}(\bigwedge_{1 \leq i \leq m_2} u'_i \bot_{w'_i} v'_i \ja \psi_0) \tai (\forall y_4 \on y_2 (\bigwedge_{1 \leq i \leq m_1} u_i \bot_{w_i} v_i \ja \phi_0)\ja x \bot y)$ (rule 2)

\item $\forall y'_{1} \on y'_{2}(\bigwedge_{1 \leq i \leq m_2} u'_i \bot_{w'_i} v'_i \ja \psi_0) \tai \forall y_4 \on y_2 \on a \on b (\bigwedge_{1 \leq i \leq m_1} u_i \bot_{w_i} v_i \ja a \bot b\ja \phi_0 \ja ab=xy)$ (rule 1 and D9)

\item $\forall y_4 \on y_2 \on a \on b (\bigwedge_{1 \leq i \leq m_1} u_i \bot_{w_i} v_i \ja a \bot b\ja \phi_0 \ja ab=xy) \tai$\\
$ \forall y'_{1} \on y'_{2}(\bigwedge_{1 \leq i \leq m_2} u'_i \bot_{w'_i} v'_i \ja \psi_0)$ (rule 2)

\item \hspace{1cm} .

\item \hspace{1cm} .

\item \hspace{1cm} .

\item $\forall y_4 \forall y'_1(\on y_3 (\bigwedge_{1 \leq i \leq m_3} u_i \bot_{w_i} v_i \ja \phi_1) \tai \on y'_{3}(\bigwedge_{1 \leq i \leq m_4} u'_i \bot_{w'_i} v'_i \ja \psi_1)$ (induction assumption)
\end{enumerate} 

where D9 is a derivation similar to D7. This concludes the claim.

Consider then the existential part of (\ref{puoliksioikein}) which is the formula
\begin{eqnarray}\label{qfpuoliksioikein}
\on y_3 ( \bigwedge_{1 \leq i \leq m_3} u_i \bot_{w_i} v_i \ja \phi_1 )\tai
 \on y'_3 (\bigwedge_{1 \leq i \leq m_4} u'_i \bot_{w'_i} v'_i  \ja \psi_1).
\end{eqnarray} 
With one application of rule \ref{rule6} we can derive from (\ref{qfpuoliksioikein}) an equivalent formula $\theta'$ of the form
\begin{align*}
&\forall \a \forall \b \on y_3\on y'_3 \on z_0\on z_1\on r [\bigwedge_{1 \leq i \leq m_3}  u_i \bot_{w_i r} v_i\ja \bigwedge_{1 \leq i \leq m_4}  u'_i \bot_{w'_i r} v'_i \ja\\ &\bigwedge_{i = 0,1}  \dep(z_i) \ja
(\neg  z_0 = z_1 \tai \a = \b) \ja ((\theta_0  \ja r = z_0) \tai (\theta_1 \ja r = z_1))].
\end{align*}
So together we can derive from (\ref{puoliksioikein}) an equivalent formula of the required form
\begin{eqnarray*}
\forall y_1 \forall y'_1 \theta'.
\end{eqnarray*}
This concludes the proof of the case $\theta = \phi \tai \psi$.

Suppose then $\theta = \phi \ja \psi$. By the induction assumption, $\phi \vdash_{\I}\phi^*$ and $\psi \vdash_{\I} \psi^*$ where $\phi^*$ and $\psi^*$ are as in (\ref{phi^*}) and (\ref{psi^*}). Now $\theta \vdash_{\I} \phi^* \ja \psi^*$, and using rule 5 and the first-order rules for $\on$ and $\ja$, it is possible to derive from $\phi^* \ja \psi^*$ an equivalent formula of the required form
\begin{equation*}\forall y_1 \forall y'_1 \on y_2 \on y'_2  (\bigwedge_{1 \leq i \leq m_1} u_i \bot_{w_i} v_i \ja \bigwedge_{1 \leq i \leq m_2} u'_i \bot_{w'_i} v'_i\ja \phi_0 \ja \psi_0).
\end{equation*}

Remembering items (3') and (4') in Lemma \ref{saannot}, it is obvious that the formulas are equivalent. This concludes the proof of Step 2.

\item[Step 3] The deductions in Step 1 and 2 (from $\phi$ to (\ref{Q}) and from $\theta$ to (\ref{theta*})) can be combined to show that
\begin{equation}\label{melk}
\phi \vdash_{\I} Q^1x_{i_1} \ldots Q^nx_{i_n} \forall y_1 \ldots \forall y_l \on y_{l+1} \ldots \on y_{l+l'}(\bigwedge_{1 \leq i \leq m} u_i \bot_{w_i} v_i \ja \theta^*).
\end{equation}

\item[Step 4] At last we can derive an equivalent formula of the form (\ref{norm}) from the formula (\ref{melk}) above. Using rule \ref{rule7} we can swap the places of existential and universal quantifiers which sit next to each other. Every swap gives us some new independence atom which we can push to conjunction 
\begin{equation*}
\bigwedge_{1 \leq i \leq m} u_i \bot_{w_i} v_i.
\end{equation*}
Pushing every universal quantifier in front of the formula and the new independence atoms to the quantifier-free part, we have a formula which is almost of the required form; every new independence atom has still variables that are not existentially quantified. We omit the proof of this part here because it is essentially the same than the proof of Step 4 in \cite{juhajouko}. Only exceptions are that rule \ref{rule7} is the independence logic version of the similar dependence logic rule and in place of $\forall$ E and $\forall$ I we use rule \ref{rule5}. After finishing this part we replace all the universally quantified variables in these new independence atoms as existentially quantified variables. This can be done easily just as we did it in Step 2 in the case of independence atoms.
\end{itemize}
Steps 1-4 show that from a sentence $\phi$ a logically equivalent sentence of the form (\ref{norm}) can be deduced.

\end{proof}

\subsection{Derivation of the approximations $\Phi^n$}
In the previous section we proved that from every sentence $\phi$ we can derive a logically equivalent sentence of the form
\begin{equation}\label{norm2}
\forall x \on y (\bigwedge_{1 \leq i \leq m} u_i \bot_{w_i} v_i \ja \theta)
\end{equation}
where $x$ and $y$ are tuples of variables; $u_i$, $v_i$ and $w_i$ are tuples of existentially quantified variables and $\theta$ is a quantifier-free first-order formula. Next we will show that the approximations $\Phi^n$ of the game expression $\Phi$ corresponding to the sentence (\ref{norm2}) can be deduced from it.

The formulas $\Phi$ and $\Phi^n$ are defined as follows.

\begin{maa}\label{appro}
Let $\phi$ be the formula (\ref{norm2}). For $j,k \in \Z$ and $1 \leq k \leq m$, we let:
\begin{itemize}
\item[-] $x$ and $x_{j,k}$ be variable tuples of same length and $y$ and $y_{j,k}$ be variable tuples of same length such that each variable occurs at most once in these tuples.
\item[-] $\theta_{j,k}=\theta(x_{j,k}y_{j,k}/xy)$ and $e^i_{j,k}=e_i(x_{j,k}y_{j,k}/xy)$ for $e \in \{u,v,w\}$.
\item[-] $ \pi^i_{n,j,k}=  \left\{\begin{array}{l l}
    \perp & \quad \textrm{if $w_i$ is empty}\\
    \neg w^i_{n,j}=w^i_{n,k} & \quad \textrm{otherwise}\\
  \end{array}\right.$
\item[-] $p_0=0$ and  $p_n = p_{n-1} +m(p_{n-1}+n+1)^2$ for $n \geq 1$.
\end{itemize} 
Also for $n \geq 1$, we define
\begin{align}
C_n &:= \bigwedge_{-n \leq i \leq p_n} \theta_{n,i},\label{C}\\
D_n &:=\bigwedge_{-n+1 \leq i \leq p_{n-1}} x_{n,i} y_{n,i} = x_{n-1,i} y_{n-1,i} ,\label{D}\\
E_n &:= \bigwedge_{\substack{1 \leq i \leq m\\ -n \leq j,k \leq p_{n-1}}} (\pi^i_{n,j,k} \tai
 \bigvee_{p_{n-1} < l \leq p_n} u^i_{n,j} v^i_{n,k} w^i_{n,j} = u^i_{n,l} v^i_{n,l} w^i_{n,l})),\label{E}
\end{align}
and let
\begin{equation}\label{Psi}
\Psi^n:=C_n \ja D_{n} \ja E_n.
\end{equation}
In the case $n=0$, we let $\Psi^0: = \theta_{0,0}$.
\begin{itemize}
\item The infinitary formula $\Phi$ is now defined as:
\begin{align*}
&\forall x_{0,0} \on y_{0,0} (\Psi^0 \ja \\
\forall x_{1,-1} \on y_{1,-1} &\on x_{1,0} \on y_{1,0} \ldots \on x_{1,p_1} \on y_{1,p_1} (\Psi^1 \ja \\
\forall  x_{2,-2} \on y_{2,-2} \on x_{2,-1} \on y_{2,-1} &\on x_{2,0} \on y_{2,0}\ldots \qquad\ldots  \on x_{2,p_2} \on y_{2,p_2} (\Psi^2 \ja \\
&\hspace{1cm}\ldots \\
&\hspace{1cm}\ldots \\
&\hspace{1.2cm} \ldots ))).
\end{align*}
\item The $n$:th approximation $\Phi^n$ of $\phi$ is defined as:
\begin{align*}
&\forall x_{0,0} \on y_{0,0} (\Psi^0 \ja \\
\forall x_{1,-1} \on y_{1,-1} &\on x_{1,0} \on y_{1,0} \ldots \on x_{1,p_1} \on y_{1,p_1} (\Psi^1 \ja \\
\forall  x_{2,-2} \on y_{2,-2} \on x_{2,-1} \on y_{2,-1} &\on x_{2,0} \on y_{2,0}\ldots \qquad\ldots  \on x_{2,p_2} \on y_{2,p_2} (\Psi^2 \ja \\
&\hspace{1cm}\ldots \\
&\hspace{1cm}\ldots \\
\forall x_{n,-n} \on y_{n,-n} \on x_{n,-n+1} \on y_{n,-n+1} \ldots\quad \ldots  &\on x_{n,0}\on y_{n,0} \ldots\hspace{1.2cm} \ldots   \on x_{n,p_n} \on y_{n,p_n} (\Psi^n) \ldots ))).
\end{align*}
\end{itemize}
\end{maa} 

Next we will show that $\phi \vdash_{\I} \Phi^n$ for natural numbers $n$.
\begin{lau}\label{approt}
Let $\phi$ and $\Phi^n$ be as in Definition \ref{appro}. Then $\phi \vdash_{\I} \Phi^n$ for all $n \geq 0$.
\end{lau}

\begin{proof}
First we define
\begin{align}
A_n &:= \bigwedge_{1 \leq i \leq m} u^i_{n,-n} \bot_{w^i_{n,-n}} v^i_{n,-n},\label{A}\\
B_n &:=  \bigwedge_{-n+1 \leq i \leq p_n} x_{n,i}y_{n,i} \subseteq x_{n,-n}y_{n,-n},\label{B}
\end{align}
and let
\begin{align}\label{Upsilon}
\Upsilon^n := A_n \ja B_n.
\end{align}
Notice that
$$\Upsilon^0 = \bigwedge_{1 \leq i \leq m} u^i_{0,0} \bot_{w^i_{0,0}} v^i_{0,0}.$$

We will prove a bit stronger claim stating that $\phi \vdash \Omega^n$ where $\Omega^n$ is defined otherwise as $\Phi^n$ except that in the last line we also have the formula $\Upsilon^n$. Hence $\Omega^n$ is of the form 
\begin{align*}
&\forall x_{0,0} \on y_{0,0} (\Psi^0 \ja \\
\forall x_{1,-1} \on y_{1,-1} &\on x_{1,0} \on y_{1,0} \ldots \on x_{1,p_1} \on y_{1,p_1} (\Psi^1 \ja \\
\forall  x_{2,-2} \on y_{2,-2} \on x_{2,-1} \on y_{2,-1} &\on x_{2,0} \on y_{2,0}\ldots \qquad\ldots  \on x_{2,p_2} \on y_{2,p_2} (\Psi^2 \ja \\
&\hspace{1cm}\ldots \\
&\hspace{1cm}\ldots \\
\forall x_{n,-n} \on y_{n,-n} \on x_{n,-n+1} \on y_{n,-n+1} \ldots\quad \ldots  &\on x_{n,0}\on y_{n,0} \ldots\hspace{1.2cm} \ldots   \on x_{n,p_n} \on y_{n,p_n} (\Upsilon^n\ja \Psi^n) \ldots ))).
\end{align*}

It is not hard to see that we can deduce $\Phi^n$ from $\Omega^n$ so proving this claim suffices. We prove the claim by induction on $n$. For $n=0$ the claim holds, since $\phi = \Omega^0$. 

Assume then that $\phi \vdash_{\I} \Omega^{h}$; we will show that  $\phi \vdash_{\I} \Omega^{n}$ where $n:=h+1$. By the induction assumption, it suffices to show that $\Omega^{h} \vdash_{\I} \Omega^{n}$. Moreover, for this it suffices to show that from the last line of $\Omega^h$, that is
\begin{align}\label{application8}
\forall x_{h,-h} \on y_{h,-h} \on x_{h,-h+1} \on y_{h,-h+1} \ldots \on x_{h,p_h} \on y_{h,p_h} (\Upsilon^h \ja \Psi^h ),
\end{align}
one can deduce
\begin{align}
&\forall x_{h,-h} \on y_{h,-h} \on x_{h,-h+1} \on y_{h,-h+1} \ldots \on x_{h,p_h} \on y_{h,p_h} (\Psi^{h} \ja \label{johdettu} \\
\forall x_{n,-n} \on y_{n,-n}& \on x_{n,-h} \on y_{n,-h}\on x_{n,-h+1} \on y_{n,-h+1}  \ldots \qquad\ldots\on x_{n,p_n} \on y_{n,p_n}(\Upsilon^{n} \ja \Psi^{n})). \notag
\end{align} 
For, we first use repeatedly rules $\on$ E, $\ja$ E and the "elimination" part of rule \ref{rule5} in order to reach \eqref{application8} from $\Omega_h$. Then having derived \eqref{johdettu}, we can do the reverse, that is, we use rules $\on$ I, $\ja$ I and the "introduction" part of rule \ref{rule5} to obtain $\Omega^n$.

We will show how to deduce \eqref{johdettu} from \eqref{application8} in two steps. 
In Step A we will deduce from  \eqref{application8}
\begin{align}
&\forall x_{h,-h} \on y_{h,-h} \on x_{h,-h+1} \on y_{h,-h+1} \ldots \on x_{h,p_h} \on y_{h,p_h} (\Psi^{h} \ja \label{i} \\
\forall x_{n,-n} \on y_{n,-n}& \on x_{n,-h} \on y_{n,-h}\on x_{n,-h+1} \on y_{n,-h+1} \ldots \on x_{n,p_h} \on y_{n,p_h}(A_n \ja B^- \ja  C^- \ja D_{n})). \notag
\end{align}
where 
\begin{align}
B^-&:=   \bigwedge_{-n+1 \leq i \leq p_h} x_{n,i}y_{n,i} \subseteq x_{n,-n}y_{n,-n},\label{B-}\\
C^- &:= \bigwedge_{-n \leq i \leq p_h} \theta_{n,i}.\label{C-}
\end{align}
Then in Step B  we will show how to deduce from \eqref{i}
\begin{align}
&\forall x_{h,-h} \on y_{h,-h} \on x_{h,-h+1} \on y_{h,-h+1}  \ldots \on x_{h,p_h} \on y_{h,p_h} (\Psi^{h} \ja \label{ii} \\
\forall x_{n,-n} \on y_{n,-n} &\on x_{n,-h} \on y_{n,-h} \on x_{n,-h+1} \on y_{n,-h+1}\ldots \quad \ldots \on x_{n,p_n} \on y_{n,p_n}(A_n \ja B^- \ja B^+ \ja\notag\\
&\hspace{7.7cm}  C^- \ja C^+ \ja D_{n} \ja E_n)). \notag
\end{align}
where 
\begin{align*}
B^+&:=   \bigwedge_{p_h+1 \leq i \leq p_n} x_{n,i}y_{n,i} \subseteq x_{n,-n}y_{n,-n},\\
C^+ &:= \bigwedge_{p_h +1 \leq i \leq p_n} \theta_{n,i}.\\
\end{align*}
Note that at the second level of \eqref{ii} we introduce new existentially quantified tuples $x_{n,p_h+1},y_{n,p_h+1},$\\ $\ldots , x_{n,p_n} , y_{n,p_n}$. Also note that \eqref{johdettu} and \eqref{ii} are identical by the definitions \eqref{Psi} and \eqref{Upsilon}, and since $B^- \ja B^+ = B_{n}$ and $C^-\ja C^+ = C_n$.
\paragraph{Step A}
In this step we will show how to deduce \eqref{i} from \eqref{application8}. Again, using back and forth rule \ref{rule5}, $\on$ E, $\on$ I, $\ja $ E and $\ja $ I we  can first duplicate $C_n$ and deduce
\begin{align*}
\forall x_{h,-h} \on y_{h,-h} \on x_{h,-h+1} \on y_{h,-h+1} \ldots \on x_{h,p_h} \on y_{h,p_h} (\Upsilon^h \ja C_h \ja \Psi^h)
\end{align*}
from \eqref{application8}. Then, interpreting $\Upsilon^h \ja C_h $ as $A$ and $\Psi^h$ as $B$, we deduce, by rule \ref{rule10},
\begin{align}
&\forall x \on y\big (\Psi^{h} \ja\label{i1} \\
\forall u \on v &\on x' \on y'(\Upsilon^h(uv/xy) \ja C_h(uv/xy)  \ja x'y'=xy \ja x'y' \sub uv)\big )\notag
\end{align}
where
\begin{align*}
x&:= x_{h,-h} ,\\
y&:= y_{h,-h}  x_{h,-h+1}  y_{h,-h+1}\ldots   x_{h,p_h}  y_{h,p_h},\\
u&:=x_{n,-n},\\
v&:= y_{n,-n}  a_{-h+1}  b_{-h+1} \ldots  a_{p_h}  b_{p_h},\\
x'&:=x_{n,-h},\\
y'&:= y_{n,-h}  x_{n,-h+1}  y_{n,-h+1}\ldots   x_{n,p_h}  y_{n,p_h}.
\end{align*}
Here the idea is that we will first show how to derive 
\begin{equation}\label{iz}
A_n \ja B^- \ja C^- \ja D_{n}
\end{equation} from
\begin{equation}\label{ix}
\Upsilon^h(uv/xy) \ja C_h(uv/xy)  \ja x'y'=xy \ja x'y' \sub uv.
\end{equation}
Then we will obtain \eqref{i} by dropping $a_i$ and $b_i$ from the quantifier prefix.

For the first objective, recall that $\Upsilon^h = A_h \ja B_h$ and note that by the definition \eqref{A}, 
\begin{equation*}
A_h(uv/xy) = A_n.
\end{equation*} Also by \eqref{B} we obtain that
\begin{equation}\label{i3}
B_h(uv/xy) =   \bigwedge_{-h+1 \leq i \leq p_h} a_ib_i \subseteq x_{n,-n}y_{n,-n}.
\end{equation}
Then by projection and permutation and $\ja$ I we  derive
\begin{equation}\label{i4}
x_{n,-h}y_{n,-h} \sub x_{n,-n}y_{n,-n} \ja \bigwedge_{-h+1 \leq i \leq p_h} x_{n,i}y_{n,i} \sub a_i b_i
\end{equation}
from $x'y'\sub uv$. Now using transitivity and the conjunction rules we obtain $B^-$, defined in \eqref{B-}, from \eqref{i3} and \eqref{i4}. For the derivation of $C^-$, first note that by the definition \eqref{C},
\begin{equation}\label{i5}
C_h(uv/xy)= \theta_{n,-n} \ja \bigwedge_{-h+1\leq i \leq p_h}\theta(a_ib_i/xy).
\end{equation}
Hence, by rule \ref{rule8} and the conjunction rules, we obtain $C^-$, defined in \eqref{C-}, from $B^-$ and the first conjunct of \eqref{i5}. Since $D_n $ is $x'y'= xy$ by the definitions (see \eqref{D} and the previous page), we have deduced \eqref{iz} from \eqref{ix}. Using this, we conclude that  \eqref{i} can be deduced from \eqref{i1} by applying back and forth rule \ref{rule5}, $\on$ E, $\on$ I, $\ja $ E and $\ja $ I. In particular, since no variable that is listed in $a_{-h+1}b_{-h+1}\ldots a_{p_h}b_{p_h}$ appears in \eqref{iz}, we may drop these variables from the quantifier prefix when applying rule $\on $ E. This concludes Step A.

\paragraph{Step B} 
In this step we will show how to deduce \eqref{ii} from \eqref{i}. For this, it suffices to show by induction on $p_h \leq q \leq p_n$ that from \eqref{i} one can derive
\begin{align}
&\forall x_{h,-h} \on y_{h,-h} \on x_{h,-h+1} \on y_{h,-h+1} \ldots\on x_{h,p_h} \on y_{h,p_h} (\Psi^{h} \ja \label{ii1} \\
\forall x_{n,-n} \on y_{n,-n} &\on x_{n,-h} \on y_{n,-h}\on x_{n,-h+1} \on y_{n,-h+1} \ldots\quad\ldots \on x_{n,q} \on y_{n,q}(A_n \ja B(q) \ja  C(q) \ja
D_{n} \ja E(q))). \notag
\end{align}
where 
\begin{align*}
B(q)&:=   \bigwedge_{-n+1 \leq i \leq q} x_{n,i}y_{n,i} \subseteq x_{n,-n}y_{n,-n},\\
C(q) &:= \bigwedge_{-n \leq i \leq q} \theta_{n,i},\\
E(q) &:= \bigwedge_{(i,j,k) \in S_q} (\pi^i_{n,j,k} \tai
 \bigvee_{p_{h} < l \leq q} u^i_{n,j} v^i_{n,k} w^i_{n,j} = u^i_{n,l} v^i_{n,l} w^i_{n,l})),
\end{align*}
and $S_q$ is the initial segment (in the lexicographic order) of $\{1, \ldots ,m\}\times \{-n ,\ldots ,p_{h}\}^2$  of size $q-p_h$. This is due to the fact that \eqref{ii1} and \eqref{ii} are identical if $q= p_n$. For this, recall that $p_n - p_h = |\{1, \ldots ,m\}\times \{-n ,\ldots ,p_{h}\}^2|$.

Next we will prove the induction claim. If $q=p_h$, then \eqref{i} and \eqref{ii1} are identical, and therefore the claim holds. Let then $p_h \leq q < p_n$, and assume the claim for $q$. For showing the claim for $q+1$, it suffices to show that from \eqref{ii1} one can deduce 
\begin{align}
&\forall x_{h,-h} \on y_{h,-h} \on x_{h,-h+1} \on y_{h,-h+1} \ldots \on x_{h,p_h} \on y_{h,p_h} (\Psi^{h} \ja \label{ii2} \\
\forall x_{n,-n} \on y_{n,-n} &\on x_{n,-h} \on y_{n,-h} \on x_{n,-h+1} \on y_{n,-h+1}\ldots\quad\ldots \on x_{n,q+1} \on y_{n,q+1}(A_n \ja B(q+1) \ja  
C(q+1) \ja \notag\\
&\hspace{10.4cm} D_{n} \ja E(q+1))). \notag
\end{align}
Again, it suffices to show that we can deduce
\begin{equation}\label{ii3}
\on x_{n,q+1} \on y_{n,q+1}(A_n \ja B(q+1) \ja  C(q+1) \ja D_{n} \ja E(q+1))
\end{equation}
from
\begin{equation}\label{ii4}
A_n \ja B(q) \ja  C(q) \ja D_{n} \ja E(q).
\end{equation}
Let $(I,J,K) \in S_{q+1}\setminus S_q$. First we obtain from $\eqref{ii4}$
\begin{equation}\label{ii5}
u^I_{n,-n} \bot_{w^I_{n,-n}} v^I_{n,-n} \ja \bigwedge_{i\in \{J,K\}} x_{n,i}y_{n,i} \sub x_{n,-n}y_{n,-n}
\end{equation}
by $\ja$ E and $\ja$ I (and reflexivity if either $J$ or $K$ is $-n$). Then we derive from \eqref{ii5} with one application of rule \ref{rule9},
\begin{align}\label{ii6}
\on x_{n,q+1} \on y_{n,q+1} (x_{n,q+1}y_{n,q+1} \sub x_{n,-n}y_{n,-n} \ja 
(\pi^I_{n,J,K} \tai  u^I_{n,J} v^I_{n,K} w^I_{n,J} = u^I_{n,q+1} v^I_{n,q+1} w^I_{n,q+1}))
\end{align}
Note that from $x_{n,q+1}y_{n,q+1} \sub x_{n,-n}y_{n,-n}$ and $\theta_{n,-n}$ we can deduce $\theta_{n,q+1}$ by rule \ref{rule8}, and from $$\pi^I_{n,J,K} \tai  u^I_{n,J} v^I_{n,K} w^I_{n,J} = u^I_{n,q+1} v^I_{n,q+1} w^I_{n,q+1}$$ we can deduce 
$$\pi^I_{n,J,K} \tai  \bigvee_{p_h < l \leq q+1}u^I_{n,J} v^I_{n,K} w^I_{n,J} = u^I_{n,l} v^I_{n,l} w^I_{n,l}$$
by $\tai$ E and $\tai$ I. Hence it is now easy to see that we can deduce \eqref{ii3} from \eqref{ii4} and \eqref{ii6} by using rule \ref{rule8} and the elimination and introduction rules of $\on$, $\ja$ and $\tai$. Therefore, it follows that \eqref{ii2} can be deduced from \eqref{ii1}. This concludes the induction proof and therefore Step B.
 We have now showed in Step A and Step B that \eqref{i} can be deduced from \eqref{application8}, and \eqref{ii} from \eqref{i}. Since \eqref{ii} and \eqref{johdettu}  are identical, we have showed that \eqref{johdettu} can be deduced from \eqref{application8}.
Therefore, we conclude that $\Omega^{h} \vdash_{\I} \Omega^{n}$ when by the induction assumption $\phi \vdash_{\I} \Omega^{n}$. This concludes the proof of Theorem \ref{approt}.

\end{proof}

\section{Back from approximations}

\begin{prop}\label{coun}
Let $\phi$ be as in (\ref{norm2}) and $\Phi$ as in Definition \ref{appro}. Then $\phi \models \Phi$ and in countable models $\Phi \models \phi$.
\end{prop}

\begin{proof}
Assume that $M \models \phi$. We show $M \models \Phi$. The truth of $\Phi$ in $M$ means that there is a winning strategy for player II in the following game

\begin{center}
$$
\begin{array}{c|cccccc}
\textrm{I} & a_{0,0} &           & a_{1,-1} & 		& \ldots\\
\hline
\textrm{II} &           &b_{0,0} &          & b_{1,-1}a_{1,0}b_{1,0}\ldots a_{1,n_1}b_{1,n_1} & 	& \ldots\\ 
\end{array}$$
\end{center}
where $a_{n,i}$, $b_{n,i}$ are tuples chosen from $M$ and player II wins if the assignment $s(x_{n,i})=a_{n,i}$, $s(y_{n,i})=b_{n,i}$ satisfies $\Psi^n$ in (\ref{appro}) for all $n$.

Let $x$ and $y$ be tuples of sizes $r$ and $r'$, respectively. Since $M \models \phi$, there is a function $F: \{\emptyset\}(M^r/x) \rightarrow \mathcal{P}(M^{r'})$ such that if $X = \{\emptyset\}(M^r/x)(F/y)$, then
\begin{equation}\label{atomhold}
M \models_{X} \bigwedge_{1 \leq i \leq m} u_i \bot_{w_i} v_i \ja \theta.
\end{equation}
 
We will now construct a winning strategy for player II recursively so that for each round $n$ the assignment $s(x)=a_{n,i}$, $s(y)=b_{n,i}$ is in $X$. 
\begin{itemize}
\item If $n=0$ and player I has played $a_{0,0}$, then player II chooses $b_{0,0}$ to be any member of $F(s)$ where $s(x)=a_{0,0}$. The assignment $s(x)=a_{0,0}$, $s(y)=b_{0,0}$ is in $X$ and $M \models_X \theta$. Thus the assignment $s(x_{0,0})=a_{0,0}$, $s(y_{0,0})=b_{0,0}$ satisfies $\theta_{0,0} = \Psi^0$.

\item Suppose then $n = h+1$ and tuples $a_{h,i}$ and $b_{h,i}$ have been played in the previous round successfully by player II and so that every assignment $s(x)=a_{h,i}$, $s(y)=b_{h,i}$ is in $X$. First player I chooses some tuple $a_{n,-n}$. Then player II chooses $b_{n,-n}$ to be some member of $F(s)$, for $s(x)=a_{n,-n}$, as above. Then II chooses $a_{n,i} = a_{h,i}$ and $b_{h,i} = b_{h,i}$ for $-h \leq i \leq p_h$. By the construction and the assumption, the assignment $s(x_{n,i})=a_{n,i}$, $s(y_{n,i})=b_{n,i}$ satisfies
\begin{align*}
\bigwedge_{-n \leq i \leq p_h} \theta_{n,i} \ja \bigwedge_{-n+1 \leq i \leq p_{n-1}} x_{n,i} y_{n,i} = x_{n-1,i} y_{n-1,i}.
\end{align*}
Now for each $a_{n,i}b_{n,i}$ which have already been played i.e. the pairs with $-n \leq i \leq p_h$, there is some assignment in $X$ corresponding to it. So for each $1 \leq i \leq m$ and $-n \leq j,k \leq p_h$, if $s(w^i_{n,j})=s(w^i_{n,k})$ (or $w_i$ is empty), then by (\ref{atomhold}), there is $t \in X$ such that $t(u_iw_i)=s(u^i_{n,j}w^i_{n,j})$ and $t(v_i)=s(v^i_{n,k})$. The set \begin{equation*}
\{(i,j,k) \mid 1 \leq i \leq m,\textrm{ } -n \leq j,k \leq p_h\}
\end{equation*} is of size $p_n - p_h$, so player II can play each remaining $a_{n,i}$ and $b_{n,i}$ as some $t(x)$ and $t(y)$ for some appropriate $t \in X$ so that the formula
\begin{align}\label{atominmuunnos}
\bigwedge_{\substack{1 \leq i \leq m\\ -n \leq j,k \leq p_h}} (\pi^i_{n,j,k} \tai 
 \bigvee_{p_h < l \leq p_n} u^i_{n,j} v^i_{n,k} w^i_{n,j} = u^i_{n,l} v^i_{n,l} w^i_{n,l})
\end{align}
holds for the assignment $s(x_{n,i})=a_{n,i}$, $s(y_{n,i})=b_{n,i}$. Then by (\ref{atomhold}) and the construction,
\begin{equation*}
\bigwedge_{p_{h+1} \leq i \leq p_n} \theta_{n,i}
\end{equation*}
holds for $s$ and thus $M \models_s \Psi^n$.
\end{itemize}
Hence there is a winning strategy for player II.

Suppose then $M$ is a countable model of $\Phi$. We let $a_{i,-i}$, $i < \omega$, be an enumeration of $M^r$. We play the game $G(M,\Phi)$ letting player I play the sequence $a_{n,-n}$ as his $n$:th move. Let $s$ be the assignment determined by the play where player II follows her winning strategy.
Let $X$ be the team consisting of the assignments $t(x) = s(x_{n,i})$, $t(y)=s(y_{n,i})$, for $n < \omega$, $-n \leq i \leq p_n$. Every formula $\theta_{n,i}$ holds for $s$, so
\begin{equation*}
M \models_X \theta.
\end{equation*}
Suppose $t,t' \in X$ and $t(w_i)=t'(w_i)$ (or $w_i$ is empty) for some $1 \leq i \leq m$. Then $t$ and $t'$ correspond to some $a_{n,j}b_{n,j}$ and $a_{n',k}b_{n',k}$. If $h =  $ max$\{n,n'\} +1$, then $a_{n,j}b_{n,j} = a_{h,j}b_{h,j}$, $a_{n',k}b_{n',k}=a_{h,k}b_{h,k}$ and $-h \leq j,k \leq p_{h-1}$. Because $s$ satisfies the last conjunct of $\Psi^h$ i.e. the formula
$$\bigwedge_{\substack{1 \leq i \leq m\\ -h \leq j,k \leq p_{h-1}}} (\pi^i_{h,j,k} \tai
 \bigvee_{p_{h-1} < l \leq p_h} u^i_{h,j} v^i_{h,k} w^i_{h,j} = u^i_{h,l} v^i_{h,l} w^i_{h,l})),$$
there is $t''\in X$ corresponding to some $a_{h,l}b_{h,l}$ such that $t''(u_iw_i)=t(u_iw_i)$ and $t''(v_i)=t'(v_i)$. Hence
\begin{equation*}
M \models_{X} \bigwedge_{1 \leq i \leq m} u_i \bot_{w_i} v_i.
\end{equation*}
The team $X$ can now be presented as $\{\emptyset\}(M^r/x)(F/y)$ for $F(t)=\{b_{n,i} \mid t(x)=a_{n,i}, n < \omega ,-n \leq i \leq p_n\}$ where $F(t)$ is always non-empty for $t \in \{\emptyset\}(M^r/x)$. Hence $M \models \phi$.
\end{proof}

Next we will define a concept of a recursively saturated model that will be important for our proof.

\begin{maa}
A model $M$ is recursively saturated if it satisfies
\begin{equation*}
\forall \yli{x}((\bigwedge_n\exists y\bigwedge_{m\leq n}\phi_m(\yli{x},y))\rightarrow
\on y\bigwedge_n\phi_n(\yli{x},y))
\end{equation*}
whenever $\{\phi_n(\yli{x},y) \mid n \in \N\}$ is recursive.
\end{maa}

The following proposition is needed.

\begin{prop}[\cite{jonjohn}]\label{onrec}
For every infinite model $M$, there is a recursively saturated countable model $M'$ such that $M \equiv M'$.
\end{prop}

Over a recursively saturated model, we can replace the game expression $\Phi$ by a conjunction of its approximations $\Phi_n$.

\begin{prop}\label{rec}
If $M$ is a recursively saturated (or finite) model, then
\begin{equation*}
M \models \Phi \leftrightarrow \bigwedge_n \Phi^n.
\end{equation*}
\end{prop}
\begin{proof}
The proof is analogous to the proof of Proposition 15 in \cite{juhajouko}.
\end{proof}

\begin{kor}\label{kor}
If $M$ is a countable recursively saturated (or finite) model, then
\begin{equation*}
M \models \phi \leftrightarrow \bigwedge_n \Phi^n.
\end{equation*}
\end{kor}
\begin{proof}
By propositions \ref{coun} and \ref{rec}.
\end{proof}

Now we can prove the main result of this article.

\begin{lau}
Let $T$ be a set of sentences of independence logic and $\phi \in$ FO. Then
\begin{equation*}
T \vdash_{\I} \phi \Leftrightarrow T \models \phi.
\end{equation*}
\end{lau}

\begin{proof}
Assume first that $T \not\vdash_{\I} \phi$. Let $T^*$ consist of all the approximations of the independence sentences in $T$. Since the approximations are provable from $T$, we must have $T^* \not\vdash_{\I} \phi$. Our deduction system covers all the first-order inference rules, so $T^* \not\vdash_{\textrm{FO}} \phi$ and thus $T^* \cup \{\neg \phi\}$ is deductively consistent in first-order logic. Let $M$ be a recursively saturated countable (or finite) model of this theory. By Corollary \ref{kor}, $M \models T \cup \{\neg \phi\}$ and thus $T \not\models \phi$.

The other direction follows from Proposition \ref{soundnessthm}.
\end{proof}

\section{Examples and open questions}

In this section we present examples for rules \ref{rule6}, \ref{rule7} and \ref{rule8} and consider some open questions regarding this topic.

\begin{esim}
Figure \ref{track}  lists  results from a track and field combined event meeting.
\begin{figure}[h]
\center
\begin{tabular}{|c|c|c|}

$\mathtt{athlete}$ & $\mathtt{event}$ & $\mathtt{result}$\\\hline
Hardee & 100m & 10.50 \\\hline
Schippers & High Jump & 1.69 \\\hline
Kazmirek & Shot Put & 14.20 \\\hline
Broersen & 200m & 24.57\\\hline
Garcia & Javelin & 66.48 \\\hline
Thiam & 800m & 2:22.98 \\\hline

$\vdots$ &$\vdots $ &$\vdots$ \\
\end{tabular} \caption{Result List\label{track}}
\end{figure}\\
In a single combined event competition, each athlete takes part in the same events. Therefore, and since the  list in Figure \ref{track} contains events from both men's and women's competitions, we may conclude that its completion, say $X$,
satisfies the following disjunction:
\begin{equation}\label{trakista}\indepc{\mathtt{ahtlete}}{\mathtt{event}} \tai \indepc{\mathtt{athlete}}{\mathtt{event}}.
\end{equation}
Now, using (essentially) rule \ref{rule6},  from \eqref{trakista} we obtain
\begin{equation}\label{trakki}
\on c(\indep{c}{\mathtt{athlete}}{\mathtt{event}} \ja (c=0\tai c=1))
\end{equation}
where $0$ and $1$ are two distinct constants.
Hence \eqref{trakki} must be true for $X$. Indeed, we can extend $X$ with a new two-valued column $\mathtt{competition}$ as illustrated in Figure \ref{esim3}. Clearly this extension satisfies
$$\indep{\mathtt{competition}}{\mathtt{athlete}}{\mathtt{event}} \ja (\mathtt{competition}=\textrm{men}\tai \mathtt{competition}=\textrm{women}).$$
\begin{figure}[h]
\center
\begin{tabular}{|c|c|c|c|}
$\mathtt{athlete}$ & $\mathtt{event}$ & $\mathtt{result}$ &$\mathtt{competition}$\\\hline
Hardee & 100m & 10.50 & men\\\hline
Schippers & High Jump & 1.69 & women\\\hline
Kazmirek & Shot Put & 14.20 & men \\\hline
Broersen & 200m & 24.57 & women\\\hline
Garcia & Javelin & 66.48 & men \\\hline
Thiam & 800m & 2:22.98 & women \\\hline

$\vdots$ &$\vdots $ &$\vdots$& $\vdots$ \\
\end{tabular} \caption{Extended Result List\label{track2}}
\end{figure}\\
\end{esim}
\begin{esim}\label{esim3}
In this example we use independence introduction (rule \ref{rule7}) in a context of uniformly continuous functions.

\begin{enumerate}
\item For every $\ep > 0$ there is $\de > 0$ such that for every $x,y$, if $|x-y| < \de$, then $|f(x)-f(y)| < \ep$.

\item Therefore, for every $\ep > 0$ and $x$ there is $\de > 0$ such that $x$ and $\de$ are independent of each other for fixed $\ep$, and for every $y$, if $|x-y| < \de$, then $|f(x)-f(y)| < \ep$.
\end{enumerate}
\end{esim}
\begin{esim}
A semester in a university is divided into two consecutive periods. In each period students enroll in various courses taught by different lecturers. Let $X_1$  be a table storing information about this from Period 1 (see Figure \ref{fall}), and let $X_2$ be that from Period 2. If in Period 1, G\"odel teaches only the course Log2, then we have that $X_1 \models \mathtt{lecturer}= \textrm{ G\"{o}del} \rightarrow \mathtt{course} =\textrm{Log2}.$ Assume that some students give up after Period 1. Then we still obtain that $X_2(\mathtt{course},\mathtt{lecturer}) \sub X_1(\mathtt{course},\mathtt{lecturer})$. Now using rule \ref{rule8}, we may conclude that
$X_2 \models \mathtt{lecturer}= \textrm{ G\"{o}del} \rightarrow \mathtt{course} =\textrm{Log2}.$

\begin{figure}
\center
\begin{tabular}{|c|c|c|}
$\mathtt{course}$ & $\mathtt{student}$ & $\mathtt{lecturer}$\\\hline
Log2 & Andersson & G\"{o}del \\\hline
Math1 & Svensson & Leibniz \\\hline
Log1 & Karlsson & Frege \\\hline
Log2 & Svensson & G\"{o}del \\\hline

$\vdots$ &$\vdots $ &$\vdots$ \\
\end{tabular} \caption{$X_1$\label{fall}}
\end{figure}
\end{esim}

In the end we have some open questions.

\begin{itemize}
\item Is there are a natural generalization of this axiomatization that would cover all the first-order consequences of independence logic \textit{formulas}? If we want to use first-order approximations in our proof, we would perhaps want to construct these approximations so that they would not contain any new relation symbols. 

\item Suppose we allow only so-called pure independence atoms i.e. atoms of the form $t_1 \bot t_2$ in our syntax. Is there a similar deductive system for this syntactical restriction (pure independence logic)? It has been showed that pure independence logic is expressively as strong as independence logic \cite{pietrojouko}.

\item Our deduction system is still relatively weak. Can we somehow improve it in order to get for example all the atomic consequences of independence logic formulas? In principle this should be possible since independence atoms and independence logic formulas can be interpreted as first-order and existential second-order logic formulas, respectively.

\end{itemize}

\section*{Acknowledgements} The author was supported by the Research Foundation of the University of Helsinki, and would like to thank Juha Kontinen and Jouko V\"a\"an\"anen for a number of useful suggestions and comments, and Fan Yang for pointing out that some of the rules need extra conditions.

\end{document}